%% file: AdLd-main-siam-arxive-mix.tex
\documentclass[onefignum,onetabnum]{siamart190516}

\input{AdLD_shared}
\input{AdLD_macros}

\ifpdf
\hypersetup{
  pdftitle={},
  pdfauthor={B. Leimkuhler, M. Sachs, G. Stoltz}
}
\fi

\newtheorem{assumption}{Assumption}
\crefname{assumption}{Assumption}{Assumptions}
\newtheorem{collorary}{Collorary}
\crefname{collorary}{Collorary}{Colloraries}

\usepackage{bbm}
\usepackage{xcolor}
\usepackage{soul}
\usepackage{comment}
\usepackage{subcaption}

\renewcommand{\leq}{\leqslant}

\renewcommand{\geq}{\geqslant}
\usepackage{todonotes}
\newcommand{\Llang}{\mathcal{L}_{\rm Lang}}
\newcommand{\ev}{\eta_\varepsilon}

\begin{document}

\maketitle

\hspace{1cm}

\emph{After publication, an error in Lemma~\ref{lem:bound_AA_Ace} came to light, which, when corrected, results in a different scaling of the lower bound~\eqref{eq:spectral:gap} on the exponential convergence rate in Theorem~\ref{thm:main}, the factor~$\gamma$ in the minimum over four terms being changed to~$\gamma/\varepsilon^2$. The convergence result, its corollaries, and the various estimates in the proof of these results have been updated accordingly in this corrected version. In particular, the results now agree with the ones recently obtained by Lo\"is Delande in~\cite{delande}, who pointed out the mistake.}

\hspace{1cm}

\begin{abstract}
  Adaptive Langevin dynamics is a method for sampling the Boltzmann--Gibbs distribution at prescribed temperature in cases where the potential gradient is subject to stochastic perturbation of unknown magnitude.   The method replaces the friction in underdamped Langevin dynamics with a dynamical variable, updated according to a negative feedback loop control law as in the Nos\'e--Hoover thermostat. Using a hypocoercivity analysis we show that the law of Adaptive Langevin dynamics converges exponentially rapidly to the stationary distribution, with a rate that can be quantified in terms of the key parameters of the dynamics. This allows us in particular to obtain a central limit theorem with respect to the time averages computed along a stochastic path. Our theoretical findings are illustrated by numerical simulations involving classification of the MNIST data set of handwritten digits using Bayesian logistic regression. 
 \end{abstract}



\section{Introduction}\label{sec:introduction}
Langevin dynamics~\cite{Pavliotis2014,LeMaBook,Lelievre2016a} is a system of stochastic differential equations which is traditionally derived as a model of a coarse-grained particle system:
\begin{equation}
\begin{aligned}\label{eq:L}
\dd \q &= \M^{-1}\p \, \dd t,\\
\dd \p &= \left ( \FF(\q) -\zeta \M^{-1} \p  \right ) \dd t   + \sigma \, \dd \W.\\
\end{aligned}
\end{equation}
Here $\q \in \RR^{n}$ represents a vector of particle positions, $\p$ is the corresponding vector of momenta, the mass matrix $\M \in \RR^{n \times n}$ is symmetric positive definite, $\FF$ is the force field (normally the negative gradient of a potential energy function $U$), $\zeta\in \RR$ is a (constant) friction coefficient, and $\sigma\in \RR$ represents the strength of coupling to the stochastic driving force defined by the Wiener increment $\dd \W$.
Although conceived as a dynamical model, Langevin dynamics is among the most versatile and popular methods for computing the statistical properties in high dimension, e.g. for molecular systems or, more recently, for many problems in high-dimensional data analysis. In this approach, the dynamical properties are ignored and the stochastic differential equations are discretized to produce ``sampling paths'' with weights approximating those associated to the (prescribed) Boltzmann-Gibbs stationary distribution with density $\rho_{\beta} \propto {\rm e}^{-\beta U}$, where, in physical settings, $\beta$ is the reciprocal of the temperature scaled by Boltzmann's constant. 

The key benefit of Langevin dynamics for sampling, compared to simpler methods such as random walk Monte Carlo, is the use it makes of the gradient of the energy function (or, in the case of data analysis, the ``log posterior''; see Section \ref{sec:data} for an example of Bayesian data analysis) which can effectively guide the collection of sampling paths, resulting in less wasted computation.  The use of Langevin dynamics as a sampling scheme is further supported by its well-understood ergodic properties (see \cite{Mattingly2002,Talay2002,cances2007theoretical,sachs2017langevin} and references therein), which ensure exponential convergence of averages to their stationary values, a property which under certain technical conditions on the potential function $U$ can be shown to carry to numerical discretization \cite{Mattingly2002,Talay2002,LeMaSt2015,bou2010long,KopecLangevin}.  

Despite these advantages of Langevin dynamics, in many applications (e.g. mixed quantum and classical molecular dynamics \cite{varnai2013tests,LeMaBook} or ``big data''  \cite{cheng2017underdamped}) the computation of the force is itself a very challenging task, thus the gradient may be effectively corrupted (due to approximation error) which leads to severe biasing of the invariant distribution. It was for precisely such cases that the Adaptive Langevin dynamics method \cite{Jones2011a,Ding2014,Shang2015,Leimkuhler2015} was created.   In this method, the friction $\zeta$ in (\ref{eq:L}) is reinterpreted as a dynamical variable, defined by a negative feedback loop control law (as in the Nos\'e-Hoover method \cite{nose1984unified}).   For concreteness, we suppose the gradient noise to be modelled by an additional stochastic process.  As discussed in \cite{Shang2015}, this can, in many cases, be interpreted as an additional (unknown) It\^{o} perturbation $\sigma_{\rm G} \, \dd \W_{\rm G}$, where $\sigma_{\rm G}^{2}$ is unknown and scales linearly with the stepsize used in the discretization of the respective continuous formulation.  The system of equations now becomes\footnote{The formulation in \cite{Shang2015} is slightly different in the form of the control law as a consequence of a linear transformation of the momenta in the presentation of the frictional force.}
\begin{equation}
\begin{aligned}\label{eq:adL:1}
\dd \q &= \M^{-1}\p \, \dd t,\\
\dd \p &= \left( -\nabla U(\q) -\zeta \M^{-1} \p  \right ) \dd t  + \sigma_{\rm G} \, \dd \W_{\rm G} + \sigma_{\rm A} \, \dd \W_{\rm A},\\
\dd \zeta &= \frac{1}{\nu}\left ( \p^{\trans}\M^{-2}\p - \frac1\beta \mathrm{Tr}\left(\M^{-1}\right) \right ) \dd t,
\end{aligned}
\end{equation}
\noindent where  $\beta$, $\sigma_{\rm G}$, $\sigma_{\rm A}$, and $\nu$ are positive scalars,
and $\W_{\rm A},\W_{\rm G}$ are two independent Wiener processess in $\RR^{n}$ with independent components (``A'' stands for ``applied'', ``G'' for ``gradient''). The auxiliary variable $\zeta$ now acts as a variable friction which restores the canonical distribution associated with the prescribed inverse temperature~$\beta$. The system (\ref{eq:adL:1}) admits the invariant probability measure (see Section~\ref{sec:main})
\begin{equation}
  \label{eq:rho_unnormalized}
  \pi(\dd\q \, \dd\p \, \dd\zeta) =Z^{-1} \exp\left(-\beta \left [\frac{\p^T\M^{-1}\p}{2} + U(\q) + \frac{\nu}{2} (\zeta-\gamma)^2\right]\right) \, \dd\q \, \dd\p \, \dd\zeta,
\end{equation}
where $Z$ is a normalization constant and 
\begin{equation}
\gamma = \frac{\beta(\sigma_{\rm G}^{2}+\sigma_{\rm A}^{2})}{2}.
\end{equation}
Assuming ergodicity, the system (\ref{eq:adL:1}) allows sampling of the Gibbs-Boltzmann probability measure with density proportional to ${\rm e}^{-\beta \left [\p^T\M^{-1}\p/2 + U(\q)\right ] }$, by marginalization, and proportional to $\rho_{\beta}$ if the momenta are ignored. 

The practical value of \cref{eq:adL:1} is that it allows simulations to be performed for complicated systems in which the potential energy function $U$ and its gradient are the consequence of substantial calculations and thus entail computational errors.  The original motivation of the article of Jones and Leimkuhler \cite{Jones2011a} was in the context of multiscale models of molecular systems where the force laws were computed using a separate numerical method and the error in this process assumed to have the character of white noise.  More recently, \cref{eq:adL:1} has been adopted in the setting of sampling of Bayesian posterior distributions in large scale data science applications \cite{Chen2014}, where the gradient noise is the consequence of incomplete calculation of the log-likelihood function based on subsampling data points from a large data set, as in the stochastic gradient Langevin dynamics method~\cite{WellingTeh2011}. In this setup the potential function $U$ corresponds to the negative log posterior density of a statistical model, i.e., for independent observations $x^{1},x^{2},\dots,x^{\ndata}$, the negative gradient of $U$ is of the form
\begin{align}\label{eq:posterior:1}
-\nabla U(\q) =  \nabla \log  p_0(\q)  + \sum_{j=1}^{\ndata} \nabla \log p(x^{j} \given \q) 
\end{align}
where $p_{0}$ is a prior density and $p( x^{j} \given \q)$ is the likelihood of the $j$-th observation. In order to avoid the linear scaling in $\ndata$ of the computational cost per evaluation of the force \eqref{eq:posterior:1}, 
the gradient force $-\nabla U(\q)$ is commonly replaced by an unbiased estimator $-\widehat{\nabla} U(\q)$ in discretizations of \eqref{eq:adL:1}. That is,
\begin{equation}\label{eq:unbiased:estimator}
-\widehat{\nabla} U(\q) = \nabla \log  \pi(\q ) +  \frac{\ndata}{m}\sum_{j \in B} \nabla \log p(x^{j} \given \q ), 
\end{equation}
where $B = \{ J_{l}\}_{l=1}^{m}, \,m\ll \ndata$ is a subset of the complete data index set --commonly referred to as a minibatch--  which is comprised of uniformly and independently sampled data point indices $J_{l} \in \{1,\dots, \ndata\}, l=1,\dots,m$, which are resampled with replacement at the beginning of every time step of a discretization of \eqref{eq:adL:1}. 

Although the presence of noise in the Adaptive Langevin model in contact with all momenta suggests hypoellipticity (as for Langevin dynamics \cite{Mattingly2002}), the way in which convergence is achieved in the Adaptive Langevin system is not straightforward.   Given a stochastic differential equation system with generator $\Lc$, let us recall that there are several well studied frameworks which can be used to derive exponential convergence rates for the semi-group ${\rm e}^{t\Lc}$ (or equivalently for the respective adjoint semi-group) in certain functional spaces.  

First, there are probabilistic techniques, which allow the derivation of exponential convergence rates of ${\rm e}^{t\Lc}$  when considered as a family of operators on weighted $L^{\infty}$ spaces (see e.g. \cite{meyn1993stability,Meyn1997,Mattingly2002}), or exponential convergence rates of the formally adjoint semi-group acting on Wasserstein metric spaces (see e.g. \cite{eberle2011reflection,eberle2019}).

Second, there also exist functional analytic proofs for exponential convergence for the case of weighted $L^{\infty}$ spaces; see \cite{Rey-Bellet2006a,Hairer2011a}.  The naive application of these methods fails in the case of \cref{eq:adL:1} due to a lack of direct stochastic control of the auxiliary  variable $\zeta$. It was only very recently shown in \cite{herzog2018exponential}, that a suitable Lyapunov function can be constructed for this system which allows to conclude exponential convergence in a  weighted $L^{\infty}$ space. 

The approach taken here is based on a third method, the alternative hypocoercivity framework of Villani \cite{Villani2009}, as further developed by Dolbeault, Mouhot, and Schmeiser \cite{DMS09,Dolbeault2015}, which can be used to derive exponential convergence rates of the semi-group when considered as a family of operators acting on subspaces of $L^{2}(\mu)$, 
 where $\mu$ 
denotes the (unique) invariant measure of the stochastic process under consideration. This technique can be applied to derive geometric convergence estimates for the underdamped Langevin equation \cite{Dolbeault2015,roussel2017spectral,iacobucci2017convergence}.   We show that this framework can also be applied directly to the system (\ref{eq:adL:1}), thus demonstrating the rapid convergence in law of the Adaptive Langevin system.  

The exponential convergence shown here has important consequences for the statistics of the samples obtained using the Adaptive Langevin method.  In particular it allows to establish a central limit theorem. Our approach also allows us to characterize the asymptotic scaling of the spectral gap of the generator associated with \eqref{eq:adL:1} when considered as an operator on the respective weighted $L^{2}$ space as ${\rm O}(\min(\gamma\nu^{-1},\gamma^{-1},\gamma\nu,\gamma^{-1}\nu^{-1}))$; a qualitative characterization of the spectral gap which is missing in the analysis in \cite{herzog2018exponential}. The scaling is confirmed in~\cite{delande} using techniques from semi-classical analysis, with additional information on leading eigenvectors and eigenfunctions of the generator of the dynamics in the small temperature regime. The derived asymptotic scaling on the lower bounds of the spectral gap allows in turn to conclude an asymptotic scaling of the asymptotic variance in the above mentioned central limit theorem as ${\rm O}(\max(\gamma,\gamma^{-1}\nu,\gamma\nu,\gamma^{-1}\nu^{-1}))$; see the discussion in Remark~\ref{rem:scaling} for an informal motivation of some terms in this asymptotic scaling. 

The remainder of this paper is structured as follows.   In 
Section \ref{sec:main} we begin by rewriting the generators of the dynamics~\eqref{eq:adL:1}, where we also check the invariance of the probability measure~\eqref{eq:rho_unnormalized}. In Subsection~\ref{sec:normalization} we normalize the dynamics~\eqref{eq:adL:1} in order to study limiting regimes associated with vanishing or diverging key parameters of the dynamics (namely the thermal mass~$\nu$ and the magnitude of the fluctuation). We can then discuss requirements of the potential energy function (Subsection~\ref{sec:assumptions}), and state the exponential convergence of the evolution semigroup in Subsection~\ref{sec:exp:conv}. The central limit theorem (CLT) is derived in Section~\ref{sec:CLT}, with upper bounds on the asymptotic variance made precise in terms of the key parameters of the dynamics. Finally, we show in Subsection~\ref{sec:Langevin_limit} that the asymptotic variance converges in the large thermal mass limit to the asymptotic variance of standard Langevin dynamics.  Section \ref{sec:numerical} contains numerical experiments assessing the relevance of parameter scalings used and demonstrating the CLT in an application to Bayesian sampling.

\section{Hypocoercivity of Adaptive Langevin dynamics}\label{sec:main}

We assume that the potential energy function~$U$ is smooth and such that $\mathrm{e}^{-\beta U(\q)}$ is integrable. In particular, \eqref{eq:rho_unnormalized} is a well defined probability measure. We first show that the probability measure~\eqref{eq:rho_unnormalized} is indeed invariant under the dynamics~\eqref{eq:adL:1}. 

The generator of~\eqref{eq:adL:1} acts on functions $\varphi = \varphi(\q,\p,\zeta)$ with $(\q,\p,\zeta) \in \mathbb{R}^{2n+1}$. It can be written as $\Ladl = \Lh + \gamma\Lo + \nu^{-1} \Lnh$ with
\begin{equation}
  \label{eq:def_Lh_Lo}
  \begin{aligned}
    \Lh & = \p^T \M^{-1}\nabla_q - \nabla U(\q)^T \nabla_p = \frac1\beta \left(\nabla_p^* \nabla_q-\nabla_q^*\nabla_p\right) = \frac1\beta \sum_{i=1}^n \partial_{p_i}^*\partial_{q_i}-\partial_{q_i}^*\partial_{p_i},\\
    \Lo & = -\p^T \M^{-1}\nabla_p + \frac1\beta \Delta_p = -\frac1\beta \nabla_p^* \nabla_p = -\frac1\beta \sum_{i=1}^n \partial_{p_i}^* \partial_{p_i}, 
\end{aligned}
\end{equation}
and
\begin{equation}
  \label{eq:def_Lnh}
  \begin{aligned}
    \Lnh & = -\nu(\zeta-\gamma)\p^T \M^{-1} \nabla_p + \left(\p^T \M^{-2} \p - \frac1\beta \mathrm{Tr}\left(\M^{-1}\right)\right) \partial_\zeta\\
    & = \frac{1}{\beta^2} \left((\partial_\zeta-\partial_\zeta^*)\nabla_p^*\nabla_p + \Delta_p^*\partial_\zeta - \Delta_p\partial_\zeta^* \right),
    \end{aligned}
  \end{equation}
  where adjoints are taken on $L^2(\pi)$. A simple computation indeed shows that $\partial_{q_i}^* = -\partial_{q_i} + \beta \partial_{q_i} U(q)$, $\partial_{p_i}^* = -\partial_{p_i} + \beta (\M^{-1} \p)_i$, $\partial_\zeta^* = -\partial_\zeta + \beta \nu (\zeta-\gamma)$ and
\[
\Delta_p^* = \Delta_p - 2\beta \p^T \M^{-1}\nabla_p + \beta^2 \left(\p^T \M^{-2} \p - \frac1\beta \mathrm{Tr}\left(\M^{-1}\right)\right).
\]
The above rewriting in terms of the elementary operators $\partial_{q_i},\partial_{p_i},\partial_\zeta$ and their adjoints immediately shows that $\Lo$ is symmetric, while $\Lh$ and $\Lnh$ are antisymmetric. Let us however emphasize that this decomposition is only used for mathematical convenience: the parameter~$\gamma$ is in fact unknown since $\sigma_G$ is not known in practice.

Another benefit of the rewriting~\eqref{eq:def_Lh_Lo}-\eqref{eq:def_Lnh} is that the actions of the operators $\Lh, \Lo, \Lnh$  make it clear that the measure with density~\eqref{eq:rho_unnormalized} is indeed invariant since $\Ac \mathbf{1} = 0$ for $\Ac \in \{ \Lh, \Lo, \Lnh\}$, so that (denoting by $C_0^{\infty}(\xDomain,\RR)$ the space of $C^\infty$ functions with compact support in~$\xDomain$)
\[
\forall \varphi \in C_0^{\infty}(\xDomain,\RR), \qquad \int_{\xDomain} \Ac \varphi\,\dd \pi = \sigma \int_{\xDomain} \varphi\, \Ac \mathbf{1} \, \dd \pi =  0,  
\]
with $\sigma = 1$ for $\Ac = \Lo$ and $\sigma = -1$ for $\Ac \in \{\Lh,\Lnh\}$, and where we relied for $\Lnh$ on the fact that elementary operators acting on different variables commute. Therefore,
\[
\forall \varphi \in C_0^{\infty}(\xDomain,\RR), \qquad \int_{\xDomain} \Ladl \varphi\,\dd \pi = 0, 
\]
which proves the invariance of~$\pi$ under the dynamics~\eqref{eq:adL:1} (see for instance~\cite{Lelievre2016a}).

\subsection{Normalization of the dynamics}
\label{sec:normalization}
To simplify the notation we let $\M = \I$. Let us however emphasize that our proofs and results can  be adapted in a straightforward way to accomodate general mass matrices. As one of our interests in this work is to understand the limiting regimes $\gamma \to 0$ or~$+\infty$ and/or $\nu \to 0$ or~$+\infty$ of the Adaptive Langevin dynamics, we also need to rescale the friction variable~$\zeta$ in order for the invariant measure to be independent of the parameter~$\nu$. More precisely, we set $\varepsilon = \sqrt{\nu}$ and consider $\xi = \sqrt{\nu}(\zeta-\gamma)$, \textit{i.e.}
\[
\zeta = \gamma + \frac{\xi}{\varepsilon}.
\]
The latter change of variables is motivated by the fact that the invariant measure~\eqref{eq:rho_unnormalized} now becomes (slightly abusing the notation~$\pi$)
\begin{equation}
  \label{eq:rho_normalized}
  \pi(\dd\q \, \dd\p \, \dd\xi) = Z^{-1} \exp\left(-\beta \left [\frac{\p^T\M^{-1}\p}{2} + U(\q) + \frac{\xi^2}{2} \right]\right) \, \dd\q \, \dd\p \, \dd\xi.
\end{equation}
Let us emphasize that this invariant probability measure does not depend on the parameters~$\gamma,\varepsilon$. The dynamics~\eqref{eq:adL:1} then becomes 
\begin{equation}
  \label{eq:adL}
  \begin{aligned}
    \dd \q &= \p \, \dd t,\\
    \dd \p &= \left( -\nabla U(\q) - \frac{\xi}{\varepsilon} \p   -\gamma \p \right ) \dd t  + \sqrt{\frac{2 \gamma}{\beta}} \dd \W,\\
    \dd \xi &= \frac{1}{\varepsilon}\left ( |\p|^2 - \frac{n}{\beta} \right ) \dd t,
  \end{aligned}
\end{equation}
where $|\p| = \sqrt{p_1^2+\dots+p_n^2}$ is the Euclidean norm of~$\p \in \mathbb{R}^n$. The generator of this SDE is
\begin{equation}
  \label{eq:generator:adL}
  \Ladl=  \Lh + \gamma\Lo + \varepsilon^{-1} \Lnh,
\end{equation}
with the above definitions~\eqref{eq:def_Lh_Lo} for $\Lh$ and $\Lo$ (upon replacing $\M$ with $\I$) and
\begin{equation}
  \label{eq:def_Lnh_rescaled}
\Lnh = \left (|\p|^{2} - \frac{n}{\beta} \right ) \partial_{\xi} - \xi\, \p^T \nabla_{p} = \frac{1}{\beta^2} \left((\partial_\xi-\partial_\xi^*)\nabla_p^*\nabla_p + \Delta_p^*\partial_\xi - \Delta_p\partial_\xi^* \right).
\end{equation}

\subsection{Assumptions and notation}
\label{sec:assumptions}

We denote by $\pi_{q}, \pi_{p}, \pi_{\xi}$ the marginals of the probability measure~\eqref{eq:rho_normalized} in the variables $\q,\p$, and $\xi$, respectively, so that $\pi(\dd \q\,\dd\p\,\dd \xi)= \pi_{q}(\dd \q) \pi_{p}(\dd \p) \pi_{\xi}(\dd \xi)$.
Further let $\norm{\,\cdot\,}_{L^{2}(\pi)}$ be the norm on the Hilbert space~$L^2(\pi)$ induced by the canonical scalar product, and denote by $L^{2}_{0}(\pi)$ the subspace of $L^2(\pi)$ of functions with vanishing mean:
\begin{equation}
L^{2}_{0}(\pi)  = \left \{ \varphi \in L^{2}(\pi) \, \left| \int_{\xDomain} \varphi \,\dd\pi =0 \right. \right \},
\end{equation}
and by $\Pi_{0} : L^{2}(\pi) \rightarrow L_{0}^{2}(\pi)$ the orthogonal projection operator onto this subspace, i.e.,
\begin{equation}
  \label{eq:def:projZero}
  \Pi_{0}\varphi = \varphi  - \int_{\xDomain} \varphi \,\dd \pi.
\end{equation}
In the remainder of this article we consider all operators as being defined on $L^{2}(\pi)$ unless explicitly specified otherwise. The associated operator norm for bounded operators on $L^2(\pi)$ is 
\[
\norm{\mathcal{T}} = \sup_{\varphi \in L^{2}(\pi) \backslash \{0\}}\frac{\norm{\mathcal{T} \varphi}_{L^2(\pi)}}{\norm{\varphi}_{L^2(\pi)}}.
\]
For an operator $\mathcal{T}$ on $L^{2}(\pi)$ with dense domain, we denote by $\mathcal{T}^{*}$ its $L^{2}(\pi)$-adjoint. Throughout the remainder of this article we assume that the potential function $U$ satisfies the following assumption. 

\begin{assumption}\label{as:poincare:U}
  The potential function $U$ is smooth, and the associated probability measure $\pi_{q}(\dd \q) = Z_q^{-1} {\rm e}^{-\beta U(\q)} \dd \q$ satisfies a Poincar\'e inequality: there exists $\qpoincare>0$ such that
\begin{equation}
\forall \varphi \in H^{1}(\pi_{q}), \qquad \norm*{\varphi - \int_{\mathbb{R}^n} \varphi\, \dd \pi_{q} }_{L^{2}(\pi_{q}) }  \leq \frac{1}{\qpoincare} \norm{\nabla \varphi}_{L^{2}(\pi_{q})}. 
\end{equation}
Moreover, there exist $c_1 > 0$, $c_2 \in [0,1)$ and $c_3 > 0$ such that
    \begin{equation}
      \label{eq:regularization condition}
      \Delta U \leq c_1 + \frac {c_2} 2 | \nabla U |^2, \quad |\nabla^2 U | \leq c_3 \left( 1 + | \nabla U | \right).
    \end{equation}
  \end{assumption}

The second condition, taken from~\cite[Section~3]{Dolbeault2015}, ensures that the operator $(1+\nabla_q^*\nabla_q)^{-1}$ is bounded from $L^2(\pi_q)$ to~$H^2(\pi_q)$. It will be used in technical estimates related to the proof of exponential convergence of the semigroup (see Lemma~\ref{lem:list_bounded_op}). 

A sufficient condition on~$U$ for~$\pi_q$ to satisfy a Poincar\'e inequality is for example the following (see \cite[Corollary 1.6]{bakry2008simple}): there exists $a \in (0,1)$, $c>0$ and $R \geq 0$ such that
\[
\forall q \in \mathbb{R}^n \text{ such that }|q|\geq R, \qquad a \beta |\nabla U(q)|^2 - \Delta U(q) \geq c.
\]
The latter condition and~\eqref{eq:regularization condition} hold for instance for potentials which behave asymptotically as $|q|^\alpha$ with $\alpha > 1$ as $\abs{q} \rightarrow \infty$.

\subsection{Exponential convergence of the law and invertibility of the generator}
\label{sec:exp:conv}

The following result states the exponential convergence in $L^{2}(\pi)$ of the semigroup~${\rm e}^{t\Ladl}$ associated with the dynamics~\eqref{eq:adL}.

\begin{theorem}
  \label{thm:main}
There exist $ \const,\overline{\lambda}$ such that, for any $\varepsilon,\gamma>0$, there is $\lambda_{\varepsilon,\gamma}>0$ for which
\begin{equation}\label{eq:L2:exp:conv}
\forall t \geq 0, \quad \forall \varphi \in L^{2}(\pi), \qquad \normLpi*{{\rm e}^{t\Ladl}\varphi - \int \varphi \,\dd \pi} \leq \const {\rm e}^{-\lambda_{\varepsilon,\gamma}} \normLpi*{ \varphi - \int \varphi \,\dd \pi},
\end{equation}
with the lower bound
\begin{equation}
  \label{eq:spectral:gap}
  \lambda_{\varepsilon,\gamma} \geq \overline{\lambda}\min\left(\frac1\gamma,\frac{1}{\gamma\varepsilon^2},\gamma\varepsilon^2,\frac{\gamma}{\varepsilon^2}\right).
\end{equation}
\end{theorem}

\cref{thm:main} immediately implies the existence of the inverse of $\Ladl$ on $L^2_0(\pi)$, and allows to obtain bounds on the norm of the inverse in terms of the parameters~$\gamma,\varepsilon$ (see \cite[Proposition 2.1]{Lelievre2016a}).

\begin{collorary}
  \label{col:invert}
  The operator $\Ladl$ considered on $L^{2}_{0}(\pi)$ is invertible and
  \[
    \Ladl^{-1} = - \int_{0}^{\infty} {\rm e}^{t\Ladl} \dd t,
    \qquad
    \norm*{\Ladl^{-1}}_{\mathcal{B}(L^2_0(\pi))} \leq \frac{ \const}{\overline{\lambda}} \max\left(\gamma,\frac{\varepsilon^2}{\gamma},\gamma\varepsilon^2,\frac{1}{\gamma\varepsilon^2}\right) .
\]
\end{collorary}
Simple computations show that some of these bounds on the resolvent are sharp. Indeed,  
\[
\Ladl \left( \gamma V+\p^T \nabla V \right) = \p^T \left(\nabla^2 V\right) \p - |\nabla V|^2 -\frac1\varepsilon \xi \p^T \nabla V,
\]
which shows that there exists $b>0$ such that $\norm*{\Ladl^{-1}}_{\mathcal{B}(L^2_0(\pi))} \geq b \gamma$ by choosing $\gamma$ large and $\varepsilon = 1$. Moreover,
\[
\Ladl\left( \gamma\varepsilon \xi + \frac{|p|^2}{2} - \frac1\gamma \p^T \nabla V \right) = -\frac{1}{\varepsilon} \xi |p|^2 + \frac{1}{\gamma\varepsilon} \p^T \nabla V - \frac1\gamma \left(\p^T \left(\nabla^2 V\right) \p - |\nabla V|^2\right),
\]
which shows that there exists $c>0$ such that $\norm*{\Ladl^{-1}}_{\mathcal{B}(L^2_0(\pi))} \geq c \gamma\varepsilon^2$ by choosing $\gamma \gg \varepsilon \gg 1$. Finally,
\[
\Ladl \left(U + \frac{|\p|^2}{2}\right) = \gamma \left[ \frac{n}{\beta}- \left(1 + \frac{\xi}{\gamma\varepsilon}\right)|\p|^2\right],
\]
which implies that there exists $a>0$ such that $\norm*{\Ladl^{-1}}_{\mathcal{B}(L^2_0(\pi))} \geq a \gamma^{-1}$ by choosing $\gamma$ small and $\gamma\varepsilon$ large. This is however weaker than the scaling~$\max(\varepsilon^2,\varepsilon^{-2})\gamma^{-1}$. It is in particular not so easy to find functions which saturate the upper bound $1/(\gamma\varepsilon^2)$ of the resolvent since this requires a careful analysis in the regime $\varepsilon \to 0$, which corresponds to a singular limit where the dominant part of the dynamics is the deterministic Nos\'e--Hoover feedback; see Remark~\ref{rmk:vanishing_eps_limit} below. We however demonstrate numerically the sharpness of the upper bound in Section~\ref{sec:spectral_gap_Galerkin}.

\bigskip

The proof of \cref{thm:main} relies on the hypercoercive framework of~\cite{DMS09,Dolbeault2015}.  The exponential decay is obtained by a Gr\"onwall inequality in a modified norm on~$L^2(\pi)$. The choice of the modified norm is motivated by the fact that $\Ladl$ is coercive in the corresponding scalar product. More precisely, we consider 
\begin{equation}
  \label{eq:modified_squared_norm}
  \mathscr{H}(\varphi) = \frac{1}{2}\normLpi{\varphi}^2 + \aa\innerLpi{\AA \varphi, \varphi },
\end{equation}
where $\AA$ is a bounded operator constructed from the antisymmetric part $\Ace:=\Lh + \varepsilon^{-1}\Lnh$ of the generator, and $\aa\in (0,1)$ is a constant. The expression of $\AA$ distinguishes whether $\varepsilon \leq 1$ or $\varepsilon \geq 1$. For $\varepsilon \in (0,1]$, the small term in~$\Ace$ is the Hamiltonian one and the expression of $\AA$ is the one suggested in~\cite{Dolbeault2015}, namely $-\left [1-\Pi \Ace^{2}\Pi\right ]^{-1}\Pi\Ace$ where $\Pi$ is the orthogonal projector on~$L^2(\pi)$ corresponding to the partial integration with respect to~$\pi_p(\dd \p)$:
\begin{equation}
  \label{eq:def_Pi}
  \left ( \Pi \varphi \right )(\q,\xi) = \int_{\mathbb{R}^n} \varphi(\q,\p,\xi) \, \pi_{p}(\dd \p).
\end{equation}
For $\varepsilon \in [1,\infty)$, the small term in $\Ace$ is the one associated with the Nos\'e--Hoover-like feedback mechanism, in which case one should rescale the generator as $\varepsilon \Ladl$ in order to avoid degeneracies as $\varepsilon \to +\infty$. Up to this multiplication by~$\varepsilon$, the regularization operator is defined as above, and therefore reads $-\varepsilon \left [1- \varepsilon^2 \Pi \Ace^{2}\Pi\right ]^{-1}\Pi\Ace$. This modification turns out to be crucial to obtain the key partial coercivity~\eqref{eq:coercivity_Pi} with the appropriate rate (see the discussion following this inequality). We therefore use the following regularization operator, which reduces to the expressions discussed above upon distinguishing $\varepsilon \leq 1$ or $\varepsilon \geq 1$:
\[
  \begin{aligned}
    \AA & := - \min\left(1,\frac1\varepsilon\right)\left [\min\left(1,\frac{1}{\varepsilon^2}\right)-\Pi \Ace^{2}\Pi\right ]^{-1}\Pi\Ace \\
    & = - \min\left(1,\frac1\varepsilon\right)\left [\min\left(1,\frac{1}{\varepsilon^2}\right)+ \Pi \left (\frac{2n}{(\beta \varepsilon)^2} \partial_{\xi}^{*}\partial_{\xi} + \frac1\beta \nabla_{q}^{*}\nabla_{q} \right ) \Pi\right ]^{-1}\Pi\Ace.
  \end{aligned}
\]
The second expression is a consequence of the following equalities:
\[
  \begin{aligned}
  \Pi \Lh^2 \Pi & = -\frac{1}{\beta^2} \sum_{i=1}^n \Pi \partial_{p_i} \partial_{p_i}^* \Pi \partial_{q_i}^* \partial_{q_i} = - \frac1\beta \nabla_{q}^{*}\nabla_{q}, \\
  \Pi \Lnh^2 \Pi & = -\frac{1}{\beta^4} \Pi \Delta_p \Delta_p^* \Pi \partial_\xi^* \partial_\xi= - \frac{2n}{\beta^2} \partial_{\xi}^{*}\partial_{\xi}, \\
  \Pi \Lnh \Lh \Pi & = \Pi \Lh \Lnh \Pi = 0,
\end{aligned}
\]
which are direct consequences of the expressions~\eqref{eq:def_Lh_Lo} and~\eqref{eq:def_Lnh_rescaled} for the generators in terms of the elementary operators $\partial_{q_i},\partial_{p_i},\partial_\xi$, as well as the following rules (which can be checked by direct computations):
\begin{equation}
  \label{eq:rules_derivatives_p_Pi}
  \partial_{p_i} \Pi = 0, \qquad \Pi \partial_{p_i}^* = 0, \qquad \partial_{p_i} \partial_{p_i}^* \Pi = \beta \Pi, \qquad \partial_{p_j}^2 \left(\partial_{p_i}^*\right)^2 \Pi = 2\beta^2 \Pi \delta_{ij}.
  \end{equation}

It can be shown that the norm of $\AA$ is bounded by 1/2 (see Lemma~\ref{lem:DNS:1}), so that $\sqrt{\mathscr{H}(\cdot)}$ defines a norm equivalent to the standard norm on~$L^2(\pi)$ for any $\aa \in (-1,1)$:
\begin{equation}
  \label{eq:equi:norm}
  \sqrt{\frac{1- |\aa|}{2}} \normLpi{\varphi} \leq \sqrt{\mathscr{H}(\varphi)} \leq \sqrt{\frac{1+|\aa|}{2}}  \normLpi{\varphi}.
\end{equation}
By polarization we can define a real valued inner product associated with~$\sqrt{\mathscr{H}(\cdot)}$ as
\begin{equation}
  \label{eq:modified_scalar_product}
\begin{aligned}
  \innera{f,g} &:= \mathscr{H}(f+g)-\mathscr{H}(f)-\mathscr{H}(g) \\
  & = \innerLpi{f,g} + \aa \innerLpi{\AA f,g} + \aa \innerLpi{\AA g,f}.
\end{aligned}
\end{equation}
Most importantly, the construction of the operator $\AA$ ensures that $\Ladl$ is coercive for the modified scalar product~\eqref{eq:modified_scalar_product}, as made precise in the following key result (see Section~\ref{sec:proof_coercivity} for the proof).

\begin{proposition}
  \label{prop:coercivity}
  There exist $\overline{a} \in (0,1)$ and $\widetilde{\lambda}>0$ such that, for any $\varepsilon,\gamma>0$ and upon choosing $\aa = \overline{a}\min(\gamma/\varepsilon,\gamma^{-1},\gamma\varepsilon^2,(\gamma\varepsilon)^{-1})$ 
  in~\eqref{eq:modified_squared_norm},
  \[
    \forall \varphi \in C^\infty_0(\xDomain) \cap L^2_0(\pi), \qquad \innera{-\Ladl \varphi,\varphi} \geq \widetilde{\lambda} \min\left(\frac{\gamma}{\varepsilon^2},\frac1\gamma,\gamma\varepsilon^2,\frac{1}{\gamma\varepsilon^2}\right) \|\varphi\|_{L^2(\pi)}^2.
  \]
\end{proposition}

\cref{thm:main} then follows from the inequality
\[
\frac{d}{dt}\left[\mathscr{H}\left(\mathrm{e}^{t \Ladl}\varphi\right)\right] = \innera{\Ladl \mathrm{e}^{t \Ladl}\varphi,\mathrm{e}^{t \Ladl}\varphi} \leq -\widetilde{\lambda} \min\left(\frac{\gamma}{\varepsilon^2},\frac1\gamma,\gamma\varepsilon^2,\frac{1}{\gamma\varepsilon^2}\right) \left\|\mathrm{e}^{t \Ladl}\varphi\right\|_{L^2(\pi)}^2,
\]
upon using the equivalence of norms~\eqref{eq:equi:norm} and resorting to a Gr\"onwall lemma. 

\begin{remark}\label{rem:scaling}
  We motivate why some of the four terms are expected in the scaling~\eqref{eq:spectral:gap} of the lower bound. First, if $\Lnh=0$, the remaining part $\Lh + \gamma \Lo$ of the generator corresponds to the underdamped Langevin equation, whose spectral gap is bounded from above by a term proportional to ${\rm O}(\min(\gamma,\gamma^{-1}))$. Similarly, in the case $\Lh=0$, it can be verified that the framework of~\cite{DMS09} can be directly applied to $\Ladl=\varepsilon^{-1}( \varepsilon\gamma\Lo +\Lnh)$ considered as an operator on~$L^2(\pi_p \pi_\xi)$, meaning that the spectral gap of this operator scales as ${\rm O}\left (\varepsilon^{-1}\min(\gamma\varepsilon,(\gamma\varepsilon)^{-1}) \right) ={\rm O}\left(\min(\gamma,\gamma^{-1}\varepsilon^{-2} \right)$. By this simple analysis, we correctly predict the terms~$\gamma$ and~$1/{\gamma \varepsilon^2}$, but we miss the term~$\gamma \varepsilon^2$ and we incorrectly predict a scaling of order~$\gamma$ instead of~$\gamma/\varepsilon^2$ in the limit~$\varepsilon \rightarrow \infty$ and~$\gamma \rightarrow 0$. The origin of these limitations on the convergence rate comes from an interaction between the Hamiltonian and Nos\'e--Hoover parts, as discussed in Remark~\ref{rmk:motivation_extra_term_scaling} below.
\end{remark}

\subsection{Proof of Proposition~\ref{prop:coercivity}}\label{sec:proof_coercivity}

In the remainder of this section, we use the shorthand notation
\[
  \ev = \min(1,\varepsilon^{-1}).
\]
We first review a few properties of the operator~$\AA$ (obtained by a straightforward adaptation of~\cite[Lemma~1]{Dolbeault2015}).

\begin{lemma}
  \label{lem:DNS:1}
  The operators $\AA$ and $\Ace\AA$ are bounded, and $\Pi \AA = \AA$. Furthermore, for any $f \in L^{2}(\pi)$,
    \[
      \normLpi{\AA f} \leq \frac{1}{2}\normLpi{(1-\Pi)f}, \qquad \normLpi{\Ace \AA f} \leq \ev\normLpi{(1-\Pi)f}.
    \]
  \end{lemma}

\begin{proof}
Consider $f\in L^{2}(\pi)$ and $u=\AA f$. Then, $(\ev^2-\Pi\Ace^{2}\Pi)u =- \ev \Pi\Ace f$. This equality already shows that $\Pi u = u$, \textit{i.e.} $\Pi \AA = \AA$. Moreover, upon taking the scalar product with~$u$, and noting that $\Pi \Ace \Pi = 0$,
\[
\begin{aligned}
& \ev^2 \normLpi{u}^{2}+\normLpi{\Ace\Pi u}^{2} = -\ev \innerLpi{\Ace \Pi u,(1-\Pi)f}\\
&\qquad \leq \ev \normLpi{\Ace \Pi u} \normLpi{(1-\Pi)f} \leq \frac{\ev^2}{4} \normLpi{(1-\Pi)f}^{2}+\normLpi{\Ace\Pi u }^{2},
\end{aligned}
\]
which implies the claimed inequalities.
\end{proof}

We now fix $\varphi \in C_0^\infty(\xDomain) \cap L^2_0(\pi)$ and evaluate
\begin{equation}\label{eq:main:ineq}
\begin{aligned}
\innera{-\Ladl \varphi, \varphi}
&= -\gamma \innerLpi{ \Lo \varphi, \varphi}
 +\aa \innerLpi{\Ace \AA \varphi,\varphi} \\
 &\ \ -\aa \innerLpi{\AA\Ace \varphi,\varphi }
 -\gamma \aa\innerLpi{\AA\Lo \varphi,\varphi },
\end{aligned}
\end{equation}
where we have used the fact that $\innerLpi{ \Ladl \varphi, \varphi}  =\innerLpi{ \Lo \varphi, \varphi}$, and $\Lo \AA = \Lo\Pi \AA = 0$. We next consider the four terms on the right-hand side of~\eqref{eq:main:ineq}:
\begin{itemize}
\item The expression~\eqref{eq:def_Lh_Lo} shows that $-\innerLpi{ \Lo \varphi, \varphi} = \beta^{-1} \|\nabla_p\varphi\|_{L^2(\pi)}^2 \geq \beta^{-1}\kappa_p^2 \|(1-\Pi)\varphi\|_{L^2(\pi)}^2$ from a Poincar\'e inequality for the Gaussian measure in~$\p$, pointwise in~$(\q,\xi)$ and then integrated with respect to~$\pi_q(\dd \q) \, \pi_\xi(\dd \xi)$ (in fact, $\kappa_p = \sqrt{\beta/m}$).
\item The term $\innerLpi{\Ace \AA \varphi,\varphi}$ is equal to $\innerLpi{\Ace \AA \varphi,(1-\Pi)\varphi}$ since $\Pi \Ace \Pi = 0$, and is therefore larger than $-\ev \|(1-\Pi)\varphi\|_{L^2(\pi)}^2$ in view of Lemma~\ref{lem:DNS:1}.
\item We decompose the term $-\innerLpi{\AA\Ace \varphi,\varphi }$ as $-\innerLpi{\AA\Ace \Pi \varphi,\varphi } - \innerLpi{\AA\Ace (1-\Pi)\varphi,\varphi }$. We first observe that the operator $\AA\Ace \Pi$ can be written, using spectral calculus, as
  \[
    \AA\Ace \Pi = f_\varepsilon\left(\mathcal{T}\right), \qquad \mathcal{T} = \Pi \left (\frac{2n}{(\beta \varepsilon)^2} \partial_{\xi}^{*}\partial_{\xi} + \frac1\beta \nabla_{q}^{*}\nabla_{q} \right ) \Pi, \qquad f_\varepsilon(x) = \frac{\ev x}{\ev^2 + x}.
  \]
  Moreover, from Poincar\'e inequalities for~$\pi_q$ and~$\pi_\xi$ (with constants~$\kappa_q$ and $\kappa_\xi = \sqrt{\beta}$),
  \[
    \mathcal{T} \geq \alpha_\varepsilon \Pi(1-\Pi_0), \qquad \alpha_\varepsilon = \min\left(\frac{2n \kappa_\xi^2}{(\beta \varepsilon)^2},\frac{\kappa_q^2}{\beta}\right), 
  \]
  so that
  \begin{equation}
    \label{eq:coercivity_Pi}
    \AA\Ace \Pi \geq \Lambda_\varepsilon \Pi(1-\Pi_0), \qquad \Lambda_\varepsilon = \frac{\ev \alpha_\varepsilon}{\ev^2 + \alpha_\varepsilon}.
  \end{equation}
  Note that $\Lambda_\varepsilon$ is of order~1 when $\varepsilon \leq 1$, and of order $\varepsilon^{-1}$ for $\varepsilon \geq 1$. It is precisely at this place that it is crucial to modify the definition of~$\AA$. Indeed, if one keeps the regularization operator $-\left [1-\Pi \Ace^{2}\Pi\right ]^{-1}\Pi\Ace$ as for $\varepsilon \leq 1$, the rate $\Lambda_\varepsilon$ would be replaced by $\alpha_\varepsilon/(1+\alpha_\varepsilon)$, which behaves as $\alpha_\varepsilon \sim \varepsilon^{-2}$ for $\varepsilon$ large. 

  The quantity $\innerLpi{\AA\Ace (1-\Pi)\varphi,\varphi }=\innerLpi{\AA\Ace (1-\Pi)\varphi,\Pi \varphi }$ can be shown to be larger than $-C_1 \max(1,\varepsilon^{-1}) \normLpi{\Pi\varphi}\normLpi{(1-\Pi)\varphi}$ upon proving that the operator $\AA\Ace(1-\Pi)$ is bounded by~$C_1 \max(1,\varepsilon^{-1})$; see Lemma~\ref{lem:bound_AA_Ace} below.
\item Finally, in order to lower bound~$\innerLpi{\AA\Lo \varphi,\varphi } = \innerLpi{\AA\Lo (1-\Pi)\varphi,\Pi \varphi }$ by $-C_2 \normLpi{\Pi\varphi}\normLpi{(1-\Pi)\varphi}$, we prove in Lemma~\ref{lem:bound:AALo} that the operator $\AA\Lo$ is uniformly bounded with respect to~$\varepsilon$ by some constant~$C_2$.
\end{itemize}
Gathering all estimates, we obtain, for $\varphi \in L^2_0(\pi)$ (so that $(1-\Pi_0)\varphi = \varphi$),
\begin{equation}
\begin{aligned}
\innera{-\Ladl \varphi, \varphi} \geq  &
\left( \frac{\gamma \kappa^{2}_{p}}{\beta} - \aa \ev\right)\normLpi{(1-\Pi)\varphi}^{2} + \aa \Lambda_\varepsilon \normLpi{\Pi \varphi}^{2} \\
& \ \ - \aa\left(C_1 \max\left(1,\frac1\varepsilon\right) + \gamma C_2\right) \normLpi{\Pi\varphi}\normLpi{(1-\Pi)\varphi} ,
\end{aligned}
\end{equation}
which can be rewritten as
\begin{equation}\label{eq:neq:XBX}
\innera{-\Ladl \varphi, \varphi} \geq X^T {\bm B}_{\varepsilon,\gamma} X, \qquad X = \begin{pmatrix} \normLpi{\Pi \varphi} \\ \normLpi{(1-\Pi) \varphi} \end{pmatrix}, \qquad {\bm B}_{\varepsilon,\gamma} = \begin{pmatrix}  B_{1,1} & \frac12 B_{1,2} \\ \frac12 B_{1,2} & B_{2,2} \end{pmatrix}, 
\end{equation}
with 
\begin{equation}
\label{eq:expressions_B_ij}
  B_{1,1} = \aa \Lambda_\varepsilon,
  \qquad
  B_{1,2} = - \aa\left(C_1 \max\left(1,\frac1\varepsilon\right) + \gamma C_2\right), 
  \qquad
  B_{2,2} = \frac{\gamma \kappa^{2}_{p}}{\beta} -\aa \ev.
\end{equation}
The result then follows from lower bounds on the smallest eigenvalue of~${\bm B}_{\varepsilon,\gamma}$, which reads
\begin{equation}
\label{eq:expression_smallest_eig_B}
\lambda({\bm B}_{\varepsilon,\gamma}) = \frac{4B_{1,1}B_{2,2}-B_{1,2}^{2}}{B_{1,1}+B_{2,2}+\sqrt{(B_{1,1}-B_{2,2})^{2}+B_{1,2}^{2}}}.
\end{equation}
The scaling of~$\aa$ as a function of~$\varepsilon,\gamma$ is obtained by requiring that the determinant
  \begin{equation}\label{eq:det:2}
    B_{1,1}B_{2,2} - \frac{B_{1,2}^2}{4} = \left(\frac{\gamma \kappa_p^2}{\beta}-\aa\right)\aa \Lambda_\varepsilon - \frac{\aa^2}{4}\left(C_1 \max\left(1,\frac1\varepsilon\right) + \gamma C_2\right)^2
  \end{equation}
is positive. We distinguish two cases:
\begin{itemize}
\item For $\varepsilon \leq 1$, $\ev = 1$ and the factor $\Lambda_\varepsilon$ is of order~1. The scaling of~$\aa$ as a function of~$\varepsilon,\gamma$ suggested by~\eqref{eq:det:2} is
  \begin{equation}
    \label{eq:scaling_a_eps_leq_1}
    \aa = \overline{a} \frac{\gamma}{(\varepsilon^{-1}+\gamma)^2} = \overline{a} \varepsilon \frac{\gamma\varepsilon}{(1+\gamma\varepsilon)^2}
  \end{equation}
  for $\overline{a}>0$ sufficiently small. We further distinguish two cases: (i) For $\gamma \varepsilon \leq 1$, the scaling~\eqref{eq:scaling_a_eps_leq_1} leads to the choice $\aa = \overline{a} \gamma \varepsilon^2$ for $\overline{a}>0$ sufficiently small, in which case the smallest eigenvalue of ${\bm B}_{\varepsilon,\gamma}$ is easily seen to be of order~$\gamma \varepsilon^2$ (since~\eqref{eq:expression_smallest_eig_B} is the ratio of a numerator of order $\gamma^2\varepsilon^2$ and a denominator of order~$\gamma$); (ii) For $\gamma \varepsilon \geq 1$, the scaling~\eqref{eq:scaling_a_eps_leq_1} leads to the choice $\aa = \overline{a}/\gamma$ for $\overline{a}>0$ sufficiently small, in which case the smallest eigenvalue of ${\bm B}_{\varepsilon,\gamma}$ is easily seen to be of order~$\min(\gamma,\gamma^{-1})$ (since the numerator in~\eqref{eq:expression_smallest_eig_B} is of order~1, while the denominator is the sum of terms proportional to~$\gamma$ and~$\gamma^{-1}$). In fact, since~$\varepsilon \leq 1$ and~$\gamma\varepsilon \geq 1$, it holds~$\gamma \geq 1$, so that the smallest eigenvalue of ${\bm B}_{\varepsilon,\gamma}$ for~$\gamma\varepsilon \geq 1$ is of order~$\gamma^{-1}$.

\item For $\varepsilon \geq 1$, the factor $\Lambda_\varepsilon$ is of order~$\varepsilon^{-1}$ and $\ev = \varepsilon^{-1}$. The scaling of~$\aa$ as a function of~$\varepsilon,\gamma$ suggested by~\eqref{eq:det:2} is
  \begin{equation}
    \label{eq:scaling_a_eps_geq_1}
    \aa = \frac{\overline{a}}{\varepsilon} \frac{\gamma}{(1+\gamma)^2}
  \end{equation}
  for $\overline{a}>0$ sufficiently small. An analysis similar to the one performed above, by further distinguishing~$\gamma \leq 1$ and~$\gamma \geq 1$, shows that the smallest eigenvalue of ${\bm B}_{\varepsilon,\gamma}$ scales as~$\aa/\varepsilon$, \emph{i.e.} $\varepsilon^{-2} \min(\gamma,\gamma^{-1})$
\end{itemize}

In conclusion, there exists $\overline{\lambda}>0$ such that the smallest eigenvalue of ${\bm B}_{\varepsilon,\gamma}$ is lower bounded by $\overline{\lambda} \min(\gamma^{-1},\gamma\varepsilon^{-2},\gamma\varepsilon^{2},(\gamma\varepsilon^2)^{-1})$.

\bigskip

We conclude this section with the proofs of the two technical lemmas used above. In these proofs, we denote by
\begin{equation}
  \label{eq:def_Lfd}
\Lfd =  \left (\ev^2 + \Pi \left[\frac{2n}{(\beta\varepsilon)^2} \partial_{\xi}^{*}\partial_{\xi} + \frac1\beta \nabla_{q}^{*}\nabla_{q}\right]\Pi \right )^{-1},
\end{equation}
so that $\AA = -\ev \Lfd \Pi \Ace$. We will repeatedly use in the proofs that $\Lfd$, when restricted to some subspace of~$L^2_0(\pi)$, behaves as $(1+ \Pi \partial_\xi^*\partial_\xi \Pi)^{-1}$ or $(1+ \Pi \nabla_q^*\nabla_q \Pi)^{-1}$. More precisely, introduce the orthogonal projectors~$P_q$ and~$P_\xi$, which correspond to a partial integration with respect to~$\pi_q(\dd \q)$ and~$\pi_\xi(\dd \xi)$ (they are the counterparts for the variables~$\q,\xi$ of the projector~$\Pi$ defined in~\eqref{eq:def_Pi}):
\begin{equation}
  \label{eq:def_P_q}
  \left ( P_q \varphi \right )(\p,\xi) = \int_{\mathbb{R}^n} \varphi(\q,\p,\xi) \, \pi_{q}(\dd \q),
  \qquad 
  \left ( P_\xi \varphi \right )(\q,\p) = \int_{\mathbb{R}} \varphi(\q,\p,\xi) \, \pi_{\xi}(\dd \xi).
\end{equation}
Note that $P_q,P_\xi$ both commute with $\Pi,\nabla_q^*\nabla_q$ and $\partial_{\xi}^{*}\partial_{\xi}$ (in fact $P_q \nabla_q^*\nabla_q = \nabla_q^*\nabla_q P_q = 0$ and $P_\xi \partial_{\xi}^{*}\partial_{\xi} = \partial_{\xi}^{*}\partial_{\xi} P_\xi = 0$) and therefore also with~$\Lfd$, and that
\begin{equation}
  \label{eq:Pi_P_operators_vanishes}
  \Pi P_q \Lh =0, \qquad \Pi P_\xi \Lnh = 0,
\end{equation}
by the invariance of the measure $\pi_q(\dd \q)\pi_p(\dd \p)$ by $\Lh$, and the invariance of $\pi_p(\dd \p)\pi_\xi(\dd \xi)$ by $\Lnh$. Moreover, $\Pi \nabla_q^*\nabla_q \Pi \geq \kappa_q^2 \Pi(1-P_q)$  from a Poincar\'e inequality for~$\pi_q$; and similarly, $\Pi \partial_\xi^*\partial_\xi \Pi \geq \kappa_\xi^2 \Pi(1-P_\xi)$  from a Gaussian Poincar\'e inequality for~$\pi_\xi$. This leads to the following result.

\begin{lemma}
  \label{lem:estimation_inverses_G_eps}
  The operators $\Lfd(1+ \Pi \nabla_q^*\nabla_q \Pi)(1-P_q)$ and $\varepsilon^{-2} \Lfd(1+ \Pi \partial_\xi^*\partial_\xi \Pi)(1-P_\xi)$ are uniformly bounded with respect to$~\varepsilon$. More precisely,
  \begin{equation}
    \label{eq:unif_bound_Lfd_q}
    \| \Lfd(1+ \Pi \nabla_q^*\nabla_q \Pi)(1-P_q) \| \leq \beta\left(1+\kappa_q^{-2}\right),
  \end{equation}
  and
  \begin{equation}
    \label{eq:unif_bound_Lfd_xi}
    \left\| \Lfd(1+ \Pi \partial_\xi^*\partial_\xi \Pi)(1-P_\xi) \right\| \leq \frac{\beta^2}{2n}\left(1+\kappa_\xi^{-2}\right)\varepsilon^2.
  \end{equation}
  Moreover, $\Lfd^{1/2}(1+ \Pi \nabla_q^*\nabla_q \Pi)^{1/2}(1-P_q)$ and $\varepsilon^{-1} \Lfd^{1/2}(1+ \Pi \partial_\xi^*\partial_\xi \Pi)^{1/2}(1-P_\xi)$ are also uniformly bounded with respect to$~\varepsilon$.
\end{lemma}

\begin{proof}
Denoting by $A_q = (1-P_q)\Pi \nabla_q^*\nabla_q \Pi(1-P_q)$, 
\[
  \begin{aligned}
    & \Lfd(1+ \Pi \nabla_q^*\nabla_q \Pi)(1-P_q) =  (1-P_q)\Lfd(1-P_q)(1+ \Pi \nabla_q^*\nabla_q \Pi)(1-P_q)\\
    & = (1-P_q + A_q)^{1/2} \left[\ev^2 + 2n(\beta\varepsilon)^{-2}\Pi \partial_{\xi}^{*}\partial_{\xi}\Pi + \beta^{-1} \Pi \nabla_q^*\nabla_q \Pi\right]^{-1}(1-P_q + A_q)^{1/2} \\
    & = (1-P_q + A_q)^{1/2} \left[\ev^2(1-P_q) + \frac{2n}{(\beta\varepsilon)^2}(1-P_q)\Pi \partial_{\xi}^{*}\partial_{\xi}\Pi(1-P_q) + \beta^{-1} A_q\right]^{-1}(1-P_q + A_q)^{1/2} ,
    \end{aligned}
\]
where all operators on the last right-hand side are considered on the subspace $(1-P_q)L^2_0(\pi)$, on which $A_q \geq \kappa_q^2$. Therefore, in the sense of symmetric operators on $(1-P_q)L^2_0(\pi)$,
\[
  \begin{aligned}
    0 & \leq \Lfd(1+ \Pi \nabla_q^*\nabla_q \Pi)(1-P_q) \leq (1 + A_q)^{1/2} \left[\ev^2 + \beta^{-1} A_q\right]^{-1}(1 + A_q)^{1/2} \leq g_\varepsilon\left(A_q\right),
    \end{aligned}
  \]
  with 
  \[
    g_\varepsilon(x) = \frac{1+x}{\ev^2+\beta^{-1}x}.
  \]
  This leads to~\eqref{eq:unif_bound_Lfd_q} since $g_\varepsilon(\kappa_q^2) \leq g_0(\kappa_q^2)$. Similar computations lead to
  \[
    \left\| \Lfd(1+ \Pi \partial_\xi^*\partial_\xi \Pi)(1-P_\xi) \right\| \leq h_0(\kappa_\xi^2) \varepsilon^2, \qquad h_\varepsilon(x) = \frac{1+x}{\min(1,\varepsilon^2)+2n\beta^{-2} x},
  \]
  which gives~\eqref{eq:unif_bound_Lfd_xi}. The estimates on $\Lfd^{1/2}(1+ \Pi \nabla_q^*\nabla_q \Pi)^{1/2}(1-P_q)$ and $\varepsilon^{-1} \Lfd^{1/2}(1+ \Pi \partial_\xi^*\partial_\xi \Pi)^{1/2}(1-P_\xi)$ are obtained in a similar way.
\end{proof} 
  
\begin{lemma}
  \label{lem:bound:AALo}
  The operator $\AA\Lo$ is uniformly bounded for $\varepsilon >0$: There exists $C_2>0$ such that $\|\AA\Lo\| \leq C_2$. 
\end{lemma}

\begin{proof}
  Since $\AA\Lo= - \ev\Lfd\Pi\Lh\Lo-  \ev\varepsilon^{-1}\Lfd\Pi\Lnh\Lo$, it suffices to prove that each operator in the right-hand side of this equality is uniformly bounded with respect to~$\varepsilon>0$. First, in view of~\eqref{eq:Pi_P_operators_vanishes}, the operator
  \[
  \begin{aligned}
    \ev\Lfd\Pi\Lh\Lo & = \ev\Lfd\Pi(1-P_q)\Lh\Lo \\
    & = \ev \Lfd(1+\Pi\nabla_q^*\nabla_q\Pi)(1-P_q)(1+\Pi\nabla_q^*\nabla_q\Pi)^{-1}\Pi(1-P_q)\Lh\Lo
  \end{aligned}
  \]
  is the product of the operator $\Lfd(1+\Pi\nabla_q^*\nabla_q\Pi)(1-P_q)$ (uniformly bounded in~$\varepsilon$ from~\eqref{eq:unif_bound_Lfd_q}) and the operator $(1+\Pi\nabla_q^*\nabla_q\Pi)^{-1}\Pi(1-P_q)\Lh\Lo$, which is bounded (see for instance~\cite[Proposition~A.3]{roussel2017spectral}); multiplied by the prefactor~$\ev \leq 1$. We next consider
  \[
  \ev\varepsilon^{-1}\Lfd\Pi\Lnh\Lo = \ev\varepsilon^{-1}\Lfd(1+\Pi\partial_\xi^*\partial_\xi\Pi)(1-P_\xi)(1+\Pi\partial_\xi^*\partial_\xi\Pi)^{-1}\Pi(1-P_\xi)\Lnh\Lo.
  \]
  Note first that the norm of the operator $\ev\varepsilon^{-1}\Lfd(1+\Pi\partial_\xi^*\partial_\xi\Pi)(1-P_\xi)$ is of order~$\min(1,\varepsilon)$ by~\eqref{eq:unif_bound_Lfd_xi}. It remains to prove that $(1+\Pi\partial_\xi^*\partial_\xi\Pi)^{-1}\Pi(1-P_\xi)\Lnh\Lo$ is bounded. We note for this that
  \[
    \begin{aligned}
      \Pi(1-P_\xi)\Lnh\Lo & = -\frac{1}{\beta^3}(1-P_\xi)\Pi \left((\partial_\xi-\partial_\xi^*)\nabla_p^*\nabla_p + \Delta_p^*\partial_\xi - \Delta_p\partial_\xi^* \right) \nabla_p^* \nabla_p \\
      & = \frac{1}{\beta^3}(1-P_\xi)\partial_\xi^*\Pi\Delta_p \nabla_p^* \nabla_p,
    \end{aligned}
  \]
  where we used~\eqref{eq:rules_derivatives_p_Pi}. The conclusion then follows from the fact that $\Pi\Delta_p \nabla_p^* \nabla_p$ is bounded (see Lemma~\ref{lem:Pi_op_p} below) as well as $\mathcal{T}_\xi = (1-P_\xi)(1+\Pi\partial_\xi^*\partial_\xi\Pi)^{-1}\Pi (1-P_\xi)\partial_\xi^*$ (by computing $\mathcal{T}_\xi \mathcal{T}_\xi^*$ and using spectral calculus together with the lower bound $\partial_\xi^*\partial_\xi \geq \kappa_\xi^2 (1-P_\xi)$ on $(1-P_\xi)L^2_0(\pi)$).

  In conclusion, $\AA\Lo= - \ev\Lfd\Pi\Lh\Lo-  \ev\varepsilon^{-1}\Lfd\Pi\Lnh\Lo$ is bounded, with an operator norm of order~$\ev + \min(1,\varepsilon)$, which is of order~1 uniformly in~$\varepsilon>0$.
\end{proof}

\begin{lemma}\label{lem:bound_AA_Ace}
  There exists $C_1>0$ such that
  \[
  \forall \varepsilon > 0, \qquad \norm{\AA\Ace(1-\Pi)} \leq C_1 \max\left(1,\frac{1}{\varepsilon}\right).
  \]
\end{lemma}

\begin{proof}
  Since $(\AA\Ace(1-\Pi))^* = (1-\Pi)\Ace \AA^{*} = \ev (1-\Pi)\Ace^{2} \Pi \Lfd$, the result is a consequence of the bound 
  \[
  \forall \varepsilon > 0, \qquad \norm{\Ace^{2} \Pi \Lfd} \leq C_1 \max\left(\varepsilon,\frac{1}{\varepsilon}\right).
  \]
  In fact, using $\Lh \Pi P_q = 0$ and $\Lnh \Pi P_\xi = 0$, 
  \begin{equation}
    \label{eq:AcA*:bound:1}
    \begin{aligned}
      \Ace^{2} \Pi \Lfd & = \Lh^{2} \Pi (1-P_q)\Lfd + \frac{1}{\varepsilon^{2}}\Lnh^{2} \Pi (1-P_\xi)\Lfd \\
      & \ \ + \frac{1}{\varepsilon}\Lnh\Lh \Pi (1-P_q)\Lfd + \frac{1}{\varepsilon}\Lh\Lnh \Pi (1-P_\xi)\Lfd.
      \end{aligned}
\end{equation}
Let us consider successively the various terms on the right-hand side. First, in view of the rules~\eqref{eq:rules_derivatives_p_Pi},
\[
\beta^2 \Lh^{2} \Pi (1-P_q)\Lfd = \sum_{i=1}^n \sum_{j=1}^n \left(\partial_{q_i,q_j}^2 (1-P_q)\Lfd\right) \left(\partial_{p_i}^*\partial_{p_j}^*\Pi\right) - \sum_{i=1}^n \sum_{j=1}^n \left(\partial_{q_i}^*\partial_{q_j} (1-P_q)\Lfd\right) \left(\partial_{p_i} \partial_{p_j}^*\Pi\right),
\]
which is a sum of bounded operators in view of Lemmas~\ref{lem:list_bounded_op} and~\ref{lem:Pi_op_p}. Similarly,
\[
  \frac{1}{\varepsilon^2} \Lnh^2 \Pi (1-P_\xi)\Lfd\Pi = \frac{1}{\beta^4}\frac{1}{\varepsilon^2}\left[(\partial_\xi-\partial_\xi^*)\partial_\xi\nabla_p^*\nabla_p + \partial_\xi^2\Delta_p^* - \partial_\xi^*\partial_\xi\Delta_p\right]\Delta_p^*\Pi (1-P_\xi)\Lfd,
\]
is a sum of bounded operators in view of Lemmas~\ref{lem:list_bounded_op} and~\ref{lem:Pi_op_p}. Consider now the terms involving both $\Lh$ and $\Lnh$. We need to introduce projectors $1-P_q$ and $1-P_\xi$ in order to rely on Lemma~\ref{lem:estimation_inverses_G_eps}. We note to this end that $\Lnh \Lh \Pi (1-P_q) = \Lnh (1-P_\xi)\Lh \Pi (1-P_q) + \Lnh P_\xi \Lh \Pi(1-P_q)$ and $\Lh \Lnh \Pi (1-P_\xi) = \Lh(1-P_q)\Lnh \Pi (1-P_\xi) + \Lh P_q\Lnh \Pi (1-P_\xi)$. Straightforward computations show that 
\[
 \frac1\varepsilon \Lnh P_\xi \Lh \Pi(1-P_q) \Lfd \varphi = - \frac1\varepsilon \xi \p^T \nabla_q \left(\Pi P_\xi (1-P_q)\Lfd\varphi\right),
\]
which is the product of two functions depending on the variables $\xi,\p$ and~$\q$, respectively, with $(\p,\xi) \mapsto \xi \p$ belonging to~$L^2(\pi_p \, \pi_\xi)$. Note also that the operators~$\partial_{q_i} \Pi P_\xi (1-P_q)\Lfd$ are uniformly bounded in~$\varepsilon>0$ in view of~\eqref{eq:unif_bound_Lfd_q}, so that finally~$\varepsilon^{-1} \Lnh P_\xi \Lh \Pi(1-P_q)$ has an operator norm of order~$\varepsilon^{-1}$. A similar reasoning shows that the operator 
\[
\frac1\varepsilon \Lh P_q\Lnh \Pi (1-P_\xi) \Lfd \varphi = -\frac2\varepsilon \p^T \nabla V \partial_\xi \left(\Pi P_q (1-P_\xi)\Lfd \varphi\right)
\]
is bounded, with an operator norm of order~$\varepsilon$ by Lemma~\ref{lem:estimation_inverses_G_eps}.
In addition, 
\[
  \begin{aligned}
    & \frac{1}{\varepsilon} \Lnh (1-P_\xi)\Lh \Pi (1-P_q) \Lfd \\
    & \quad = \frac{1}{\beta^3}\left[\frac{1}{\varepsilon}(\partial_\xi-\partial_\xi^*)(1-P_\xi)\nabla_p^*\nabla_p + \frac{1}{\varepsilon}\partial_\xi(1-P_\xi)\Delta_p^* - \frac{1}{\varepsilon}\partial_\xi^*(1-P_\xi)\Delta_p\right] \nabla_p^* \nabla_q (1-P_q)\Lfd,
  \end{aligned}
  \]
  and
  \[
  \frac{1}{\varepsilon} \Lh(1-P_q)\Lnh \Pi (1-P_\xi) \Lfd = \frac{1}{\beta^3} \frac{1}{\varepsilon}(\nabla_p^*\nabla_q-\nabla_q^* \nabla_p)(1-P_q)\partial_\xi (1-P_\xi)\Lfd \Delta_p^*\Pi,
\]
are sums of bounded operators in view of Lemma~\ref{lem:list_bounded_op}. Therefore, $\varepsilon^{-1}\Lnh\Lh \Pi (1-P_q)\Lfd$ and $\varepsilon^{-1} \Lh\Lnh \Pi (1-P_\xi)\Lfd$ are bounded operators with operator norms respectively of order~$\max(1,\varepsilon^{-1})$ and~$\max(1,\varepsilon)$. This finally gives the claimed result. 
\end{proof}

\begin{remark}
  \label{rmk:motivation_extra_term_scaling}
  Among the various terms in the decomposition of $\AA\Ace(1-\Pi)$ we consider in the proof of Lemma~\ref{lem:bound_AA_Ace}, the only one which is not bounded as $\varepsilon \to 0$ is $\varepsilon^{-1} \Lnh P_\xi \Lh (1-P_q)\Pi\Lfd$. This term arises from the interaction between the Hamiltonian and Nos\'e--Hoover parts of the dynamics, and is responsible for the factor $\max(1,\varepsilon^{-1})$ in the expression of $B_{1,2}$ in~\eqref{eq:expressions_B_ij}, which itself leads to the extra term $\gamma\varepsilon^2$ in the scaling of the lower bound of Proposition~\ref{prop:coercivity}.
\end{remark}  

Note that, crucially, operators in the~$\xi$ variable in the computations of the proof of Lemma~\ref{lem:bound_AA_Ace} always appear with a prefactor~$\varepsilon^{-1}$. The fact that this is the correct scaling for the boundedness of these operators comes from the following result.

\begin{lemma}
  \label{lem:list_bounded_op}
  The operators $\partial_{q_i,q_j}^2 (1-P_q)\Lfd$, $\partial_{q_i}^*\partial_{q_j} (1-P_q)\Lfd$, $\varepsilon^{-1}\partial_{\xi}^*(1-P_\xi)\partial_{q_i}(1-P_q)\Lfd$, $\varepsilon^{-1} \partial_{q_i}(1-P_q)\partial_\xi(1-P_\xi)\Lfd$, $\varepsilon^{-1} \partial_{q_i}^*(1-P_q)\partial_\xi(1-P_\xi)\Lfd$, $\varepsilon^{-2}\partial_\xi^2(1-P_\xi)\Lfd$, $\varepsilon^{-2}\partial_\xi^*\partial_\xi(1-P_\xi)\Lfd$ are uniformly bounded with respect to~$\varepsilon>0$.
\end{lemma}

\begin{proof}
  Consider for instance $\partial_{q_i,q_j}^2(1-P_q)\Lfd$. It is sufficient by Lemma~\ref{lem:estimation_inverses_G_eps} to prove that $\partial_{q_i,q_j}^2 (1-P_q)(1+\Pi\nabla_q^*\nabla_q\Pi)^{-1}$ is bounded, and in fact that operators of the form~$\mathcal{T}_i = \partial_{q_i}(1-P_q)(1+\Pi\partial_{q_i}^*\partial_{q_i}\Pi)^{-1/2}$ and $\partial_{q_i}^2(1-P_q)(1+\Pi\nabla_q^*\nabla_q\Pi)^{-1}$ are bounded. The first statement is clear by calculating $\mathcal{T}_i^*\mathcal{T}_i$ and using spectral calculus; while for the second one we use~\cite[Section~3]{Dolbeault2015}. Similar reasonings can be used to bound $\partial_{q_i}^*\partial_{q_j}(1-P_q)\Lfd$. Bounds on $\varepsilon^{-2}\partial_\xi^2(1-P_\xi)\Lfd$, $\varepsilon^{-2}\partial_\xi^*\partial_\xi(1-P_\xi)\Lfd$ are obtained in a similar way, considering the specific case of quadratic potentials in~$\xi$ (so that estimates similar to those of~\cite[Section~3]{Dolbeault2015} hold in the~$\xi$ variable).

  Consider next $\varepsilon^{-1}\partial_{\xi}(1-P_\xi)\partial_{q_i}^*(1-P_q)\Lfd = \mathcal{T}_\xi R_{q_i} S_{q,\xi}$ with
  \[
    S_{q,\xi} = \varepsilon^{-1}(1-P_\xi)(1+\Pi\partial_\xi^*\partial_\xi\Pi)^{1/2}(1-P_q)(1+\Pi\nabla_q^* \nabla_q\Pi)^{1/2}\Lfd
  \]
  uniformly bounded in~$\varepsilon$ by Lemma~\ref{lem:estimation_inverses_G_eps}, $\mathcal{T}_\xi = \partial_{\xi}(1-P_\xi)(1+\Pi\partial_\xi^*\partial_\xi\Pi)^{-1/2}$ bounded by considering $\mathcal{T}_\xi^* \mathcal{T}_\xi$ and resorting to spectral calculus, and $R_{q_i} = \partial_{q_i}^*(1-P_q)(1+\Pi\nabla_q^* \nabla_q\Pi)^{-1/2}$. To prove that the latter operator is bounded, we write it as the sum of $-\partial_{q_i}(1-P_q)(1+\Pi\nabla_q^* \nabla_q\Pi)^{-1/2}$ (which is bounded by the same reasoning as the one used to prove that $\mathcal{T}_\xi$ is bounded) and $\beta \partial_{q_i}V (1-P_q)(1+\Pi\nabla_q^* \nabla_q\Pi)^{-1/2}$, which is bounded in view of the inequality
  \[
    \| |\nabla V|h \|_{L^2(\pi_q)} \leq C \left( \|h\|_{L^2(\pi_q)} + \| \nabla h \|_{L^2(\pi_q)}\right)
  \]
  provided by~\cite[Lemma~A.24]{Villani2009}. The boundedness of $\varepsilon^{-1} \partial_{q_i}(1-P_q)\partial_\xi(1-P_\xi)\Lfd$ and $\varepsilon^{-1} \partial_{q_i}(1-P_q)\partial_\xi^*(1-P_\xi)\Lfd$ follows by similar arguments.
\end{proof}

The proof of the following lemma is obtained by straightforward computations based on integration by parts in the integral involved in the definition of~$\Pi$.

\begin{lemma}
  \label{lem:Pi_op_p}
  For any $\alpha_1,\alpha_2,\alpha_3 \in \mathbb{N}$ and $i,j,k \in \{1,\dots,n\}$, the operators $\Pi \partial_{p_i}^{\alpha_1} \left(\partial_{p_j}^*\right)^{\alpha_2} \partial_{p_k}^{\alpha_3}$ are bounded (and so are their adjoints on~$L^2(\pi_p)$ and~$L^2(\pi)$). In particular, $\partial_{p_i}^* \partial_{p_j}^* \Pi$ and $\partial_{p_i} \partial_{p_j}^* \Pi$ are bounded.
\end{lemma}

\section{Pathwise ergodicity and functional central limit theorem}
\label{sec:CLT}

Consider, for $\varphi \in L^{1}(\pi)$ given, the trajectory average of $\varphi$ evaluated along a realization of the solution of the SDE~\cref{eq:adL}:
\begin{equation}
  \label{eq:estimator}
\widehat{\varphi}_{t} := \frac{1}{t}\int_{0}^{t} \varphi(\q_s,\p_s,\xi_s) \, \dd s.
\end{equation}
The almost-sure convergence of this estimator to $ \EE_{\pi}(\varphi)$ holds by the results of~\cite{kliemann1987recurrence} since the dynamics admits an invariant probability measure with a positive density, and the generator is hypoelliptic~\cite{Hor67}. The latter property follows from the following computations on commutators: $[\Lh,\partial_{p_i}] = -\partial_{q_i}$, and $[\Lnh,\partial_{p_i}] = -2p_i\partial_{\xi}+\xi \partial_{p_i}$ so that $[[\Lnh,\partial_{p_i}],\partial_{p_i}] = 2\partial_\xi$.

In fact, by the results from~\cite{Bhattacharya1982}, a natural central limit theorem is a consequence of the boundedness of the inverse of the generator obtained in~\cref{col:invert}.

\begin{collorary}[Central limit theorem for AdL]\label{col:clt}
Consider $\varphi \in L^{2}(\pi)$. Then 
\begin{equation}\label{eq:clt}
\sqrt{t} \left ( \widehat{\varphi}_{t}  -  \EE_{\pi} \varphi \right ) \xrightarrow[t \to +\infty]{\mathrm{law}}  \mathcal{N}(0,{\sigma}_{\varepsilon,\gamma}^{2}(\varphi)),
\end{equation}
where the asymptotic variance reads
\[
  {\sigma}_{\varepsilon,\gamma}^{2}(\varphi) =   2 \int_{\xDomain}   \left ( -\Ladl^{-1}\Pi_{0} \varphi \right )\Pi_{0} \varphi \, \dd \pi. 
\]
\end{collorary}

\cref{col:invert} provides the following bounds on the asymptotic variance:
\begin{equation}
\label{eq:scaling_asym_var}
0 \leq {\sigma}_{\varepsilon,\gamma}^{2}(\varphi) \leq  \frac{ 2\const \|\varphi\|_{L^2(\pi)}^2}{\overline{\lambda}} \max\left(\gamma,\frac{\varepsilon^2}{\gamma},\gamma\varepsilon^2,\frac{1}{\gamma\varepsilon^2} \right).
\end{equation}
This inequality shows that integration times of order $t = \tau \max\left(\gamma,\gamma^{-1}\varepsilon^2,\gamma\varepsilon^2,(\gamma\varepsilon^2)^{-1} \right)$ should be considered in order for the estimator~\eqref{eq:estimator} to have a variance of order~$1/\tau$.

\subsection{Langevin limit $\varepsilon \to +\infty$}
\label{sec:Langevin_limit}

We consider in this section the convergence of the asymptotic variance in the limit when $\varepsilon \to +\infty$, which should be thought of as being somewhat similar to overdamped limits of Langevin dynamics. We do not consider the regime $\varepsilon \to 0$ which is a mathematically a singular limit (see however Remark~\ref{rmk:vanishing_eps_limit} below), and is also not a regime which is numerically convenient because of the stiffness of the resulting dynamics, which typically calls for integration schemes with timesteps of order~$\varepsilon$ (or the construction of dedicated numerical schemes based on averaging ideas for instance).

In the limit $\varepsilon \to +\infty$, for a given test function $\varphi \in C^\infty_0(\xDomain)$, the function $\Ladl \varphi$ converges to $\Llang\varphi$ where $\Llang = \Lh + \gamma \Lo$ is the generator of the standard underdamped Langevin dynamics. To understand the behavior of the limiting asymptotic variance, we restrict ourselves to functions of $(\q,\p)$ only, since the variable~$\xi$ evolves very slowly and should therefore not be of interest. Since the slow convergence to equilibrium is due to the relaxation of the~$\xi$ variable in the regime $\varepsilon \to +\infty$, we expect that restricting the attention to such observables allows the variance to remain bounded. In fact, the following result holds (see Section~\ref{sec:proof_asymptotic_analysis} for the proof).

\begin{proposition}
  \label{prop:asymptotic_analysis}
  Fix $\gamma > 0$. Assume that $U$ satisfies Assumption~\ref{as:poincare:U}, is semi-convex (there exists a bounded smooth function $U_1$ with bounded derivatives and a smooth convex function $U_2$ such that $U = U_1+U_2$), grows at most polynomially at infinity and its derivatives as well, and that there exist $K>0$,$R\in \mathbb{R}$ and $a \in (0,1)$ such that 
  \[
   \frac12 \q^T \nabla U(\q) \geq a U(\q) + \gamma^2 \frac{a(2-a)}{8(1-a)}|\q|^2-K, \qquad U(\q) \geq R |\q|^2.
  \]
  Consider a smooth function $\varphi = \varphi(\q,\p)$ growing at most polynomially in~$(\q,\p)$ and whose derivatives grow at most polynomially. Then there exists $C>0$ (depending on~$\gamma$, $\varphi$) such that the asymptotic variance ${\sigma}_{\varepsilon,\gamma}^{2}(\varphi)$ defined in Corollary~\ref{col:clt} satisfies
  \[
    \forall \varepsilon \geq 1, \qquad \left|{\sigma}_{\varepsilon,\gamma}^{2}(\varphi)-{\sigma}_{\infty,\gamma}^{2}(\varphi)\right| \leq \frac{C}{\varepsilon},
  \]
  where ${\sigma}_{\infty,\gamma}^{2}(\varphi)$ involves only asymptotic variances of underdamped Langevin dynamics. More precisely,
  \[
    {\sigma}_{\infty,\gamma}^{2}(\varphi) = \frac{2}{\beta}\left(\gamma\|\nabla_p \Phi_0\|^2_{L^2(\pi_q\pi_p)} - \gamma\frac{\langle \nabla_p \Phi_{-1}, \nabla_p \Phi_0\rangle_{L^2(\pi_q\pi_p)}^2}{\|\nabla_p \Phi_{-1}\|^2_{L^2(\pi_q\pi_p)}} + \frac{ \beta^2 \langle \Phi_{-1},\Lh \Phi_0\rangle_{L^2(\pi_q\pi_p)}^2}{\gamma \|\nabla_p \Phi_{-1}\|^2_{L^2(\pi_q\pi_p)}}\right)
  \]
  where $\Phi_0 = -\Llang^{-1}\Pi_0\varphi$ and $\Phi_{-1} = -\Llang^{-1}\left(\p^2 -\frac{n}{\beta}\right)$.
\end{proposition}

Note that the first term on the right-hand side of the expression of $ {\sigma}_{\infty,\gamma}^{2}(\varphi)$ corresponds to the asymptotic variance of a standard underdamped Langevin dynamics. The Nos\'e--Hoover like thermostat adds two terms in the large $\varepsilon$ limit, one nonpositive and one nonnegative, so that it is not clear in general whether  $ {\sigma}_{\infty,\gamma}^{2}(\varphi)$ is larger than $2\gamma\beta^{-1}\|\nabla_p \Phi_0\|^2_{L^2(\pi_q\pi_p)}$. Overall, it however still holds ${\sigma}_{\infty,\gamma}^{2}(\varphi) \geq 0$ as expected since a Cauchy-Schwarz inequality shows that the sum of the two first terms in the brackets on the right-hand side is indeed nonnegative. 

The extra conditions on the potential, taken from~\cite{KopecLangevin}, are satisfied for potentials growing at infinity as $|\q|^\alpha$ with $\alpha>2$. They ensure that $\Llang^{-1}$ stabilizes the vector space of smooth functions of~$(\q,\p)$ with mean zero with respect to~$\pi_q \, \pi_p$, growing at most polynomially at infinity, and whose derivatives grow at most polynomially at infinity. 

It is in fact possible to write an expansion in inverse powers of~$\varepsilon$ for the difference ${\sigma}_{\varepsilon,\gamma}^{2}(\varphi)-{\sigma}_{\infty,\gamma}^{2}(\varphi)$, and in particular to make precise the leading order term in this expansion. We however refrain from doing so because the expressions are cumbersome. Note also that the proof of Proposition~\ref{prop:asymptotic_analysis} allows to write the action of $\Ladl^{-1}$ on $L^2_0(\pi)$ at leading order~$\varepsilon^{-2}$ (in a similar fashion to the results presented in~\cite[Theorem~2.5]{LeMaSt2015}, which provides an expansion of the resolvent of the generator of the underdamped Langevin dynamics in inverse powers of~$\gamma$); see Remark~\ref{rmk:leading_order_resolvent_AdL}.

\begin{remark}
  \label{rmk:vanishing_eps_limit}
  In the limit $\varepsilon\to 0$, the dynamics~\eqref{eq:adL} behaves at dominant order as the following ordinary differential equation:
  \[
    \begin{aligned}
      \dd \q &= 0,\\
      \dd \p &= - \frac{\xi}{\varepsilon} \p \, \dd t, \\
      \dd \xi &= \frac{1}{\varepsilon}\left ( |\p|^2 - \frac{n}{\beta} \right ) \dd t.
    \end{aligned}
  \]
  The only equilibrium points correspond to $|\p|^2 = n \beta^{-1}$ and $\xi=0$. A simple computation shows that
  \[
    \Phi(q,\p,\xi) = \xi^2 + |\p|^2 - \frac{2n}{\beta} \ln |\p|^2
  \]
  is an invariant of the dynamics. It is therefore expected that~\eqref{eq:adL} corresponds to a fast averaging on the level sets of~$\Phi$, with a superimposed slow variation of the values of~$\Phi$ induced by the Langevin part of the dynamics. Since the dynamics is at leading order a dynamics on the two one-dimensional variables~$P = |\p|^2$ and~$\xi$ only, it might be possible to adapt the techniques from~\cite{PV08} in order to determine the dominant behavior of the asymptotic variance in the regime~$\varepsilon \to 0$.
\end{remark}

\subsection{Proof of Proposition~\ref{prop:asymptotic_analysis}}
\label{sec:proof_asymptotic_analysis}

The idea of the proof is to construct an approximate solution $\psi_\varepsilon$ to the Poisson equation $-\Ladl \phi_\varepsilon = \Pi_0 \varphi$, using asymptotic analysis. The scaling of the resolvent $-\Ladl$ as given by Corollary~\ref{col:invert} suggests that, in the limit $\varepsilon\to+\infty$,  
\begin{equation}
  \label{eq:formal_expansion}
  \psi_\varepsilon = \varepsilon^2 \Psi_{-2} + \varepsilon \Psi_{-1} + \Psi_0 + \varepsilon^{-1} \Psi_1 + ...
\end{equation}
The various functions in~\eqref{eq:formal_expansion} formally satisfy, by identifying powers of~$\varepsilon$,
\[
  \begin{aligned}
    -\Llang \Psi_{-2} & = 0, \qquad -\Llang \Psi_{-1} = \Lnh \Psi_{-2}, \qquad -\Llang \Psi_0 = \Pi_0 \varphi + \Lnh \Psi_{-1}, \\
    -\Llang \Psi_{i} & = \Lnh \Psi_{i-1} \quad \textrm{ for } i \geq 1.  
  \end{aligned}
\]
The strategy of the proof is to construct the leading order terms~$\Psi_{-2},\Psi_{-1},\dots,\Psi_{2} \in L^2_0(\pi)$ in order to obtain some approximate solution~$\psi_\varepsilon$ (obtained by a truncation of~\eqref{eq:formal_expansion}), and then to use resolvent estimates to conclude that $\phi_\varepsilon-\psi_\varepsilon$ is small.  

We will  repeatedly use the fact that the unique solution~$G$ of $-\Llang G = g$ for $g$ a smooth function with average~0 with respect to~$\pi_p(\dd \p)\pi_q(\dd \q)$ growing at most polynomially at infinity and whose derivatives also grow at most polynomially at infinity, is a well defined smooth function, which grows at most polynomially at infinity and whose derivatives also grow at most polynomially at infinity (by the results of~\cite{KopecLangevin}). 

\paragraph{Construction of the leading order terms in the expansion}
The equation $-\Llang \Psi_{-2} = 0$ shows that $\Psi_{-2}(\q,\p,\xi) = f_{-2}(\xi)$. Next, $-\Llang \Psi_{-1} = \Lnh \Psi_{-2} = (\p^2 -n\beta^{-1})f_{-2}'(\xi)$, so that
\[
  \Psi_{-1}(\q,\p,\xi) = f_{-2}'(\xi)\Phi_{-1}(\q,\p) + f_{-1}(\xi), \qquad \Phi_{-1}(\q,\p) = -\Llang^{-1}\left(\p^2 -\frac{n}{\beta}\right).
\]
The equation for $\Psi_0$ then reads
\[
-\Llang \Psi_0 = \Pi_0 \varphi + f_{-2}''(\xi)\left(\p^2 -\frac{n}{\beta}\right)\Phi_{-1} + f_{-1}'(\xi) \left(\p^2 -\frac{n}{\beta}\right) - \xi f_{-2}'(\xi) \p^T \nabla_p \Phi_{-1}.
\]
The solvability condition for this equation is that the right-hand side has average~0 with respect to the probability measure~$\pi_q(\dd \q)\pi_p(\dd \p)$.  Integration by parts shows that, for any test function~$\phi$,  
\[
\int_{\mathbb{R}^n} \p^T \nabla_p \phi \, \dd \pi_p = \beta \int_{\mathbb{R}^n} \phi\left(\p^2 -\frac{n}{\beta}\right)  \dd \pi_p,
\]
so that the solvability condition reads
\begin{equation}
  \label{eq:solvability_f_-2}
  a\mathcal{L}_{{\rm eff},\xi} f_{-2} = -\int_{\mathbb{R}^{2n}} \Pi_0 \varphi \, \dd\pi_p \, \dd \pi_q = 0, \qquad a = \int_{\mathbb{R}^{2n}} \left(\p^2 -\frac{n}{\beta}\right)\Phi_{-1} \, \dd \pi_p \, \dd \pi_q \geq 0,
\end{equation}
where $\mathcal{L}_{{\rm eff},\xi}$ is the generator of an effective Ornstein--Uhlenbeck process acting on functions $u = u(\xi)$ as $\mathcal{L}_{{\rm eff},\xi}u = u'' - \beta \xi u'$. 
In fact $a > 0$ since $a = \gamma\beta^{-1}\|\nabla_p \Phi_{-1}\|^2 = 0$ would imply that $\Phi_{-1}$ is constant in $\p$, which is in contradiction to the definition of $\Phi_{-1}$ because
\[
\left(-\Llang\Phi_{-1}\right)(\q,\p)= \p^{\trans}  \cdot \nabla_{q}\Phi_{-1}(\q) \neq   \left(\p^2 -\frac{n}{\beta}\right).
\]
The fact that $a$ is nonzero implies that the first equality in \eqref{eq:solvability_f_-2} holds if and only if $f_{-2} = 0$, so that $\Psi_{-2} = 0$ and $\Psi_{-1} = f_{-1}$. Moreover,
\[
  \Psi_0(\q,\p,\xi) = \Phi_0(\q,\p) + f_{-1}'(\xi) \Phi_{-1}(\q,\p) + f_0(\xi), \qquad \Phi_0 = -\Llang^{-1}\Pi_0\varphi.
\]

\begin{remark}
  \label{rmk:leading_order_resolvent_AdL}
  The equality~\eqref{eq:solvability_f_-2} shows that the action of leading order of the resolvent for Adaptive Langevin for functions $\varphi \in L^2_0(\pi)$ is $a^{-1} \varepsilon^2 \mathcal{L}^{-1}_{{\rm eff},\xi}\Pi P_q\varphi $ (with $P_q$ defined in~\eqref{eq:def_P_q}). 
\end{remark}

The condition at next order reads
\[
-\Llang \Psi_1 = \Lnh \Psi_0 = -\xi \p^T \nabla_p \Phi_0 - \xi f_{-1}' \p^T \nabla_p \Phi_{-1} + \left(\p^2 -\frac{n}{\beta}\right)\left[f_0' + f_{-1}'' \Phi_{-1}\right].
\]
The solvability condition reads $a \mathcal{L}_{{\rm eff},\xi} f_{-1} = \xi b_0$ with $b_0 = \Pi P_q (\p^T \nabla_p \Phi_0)$, so that $f_{-1}(\xi) = -\xi b_0/(a \beta)$, and
\[
\Psi_1(\q,\p,\xi) = f_0'(\xi) \Phi_{-1}(\q,\p) + \xi \Phi_1(\q,\p) + f_{1}(\xi), \qquad \Phi_1 = -\Llang^{-1}\left(\frac{b_0}{a \beta} \p^T \nabla_p \Phi_{-1}-\p^T \nabla_p \Phi_{0}\right).
\]
Next,
\[
-\Llang \Psi_2 = \Lnh \Psi_1 = - \xi f_0' \p^T \nabla_p \Phi_{-1} - \xi^2 \p^T \nabla_p \Phi_1 + \left(\p^2 -\frac{n}{\beta}\right)\left[\Phi_1 + f_1' + f_0'' \Phi_{-1}\right],
\]
for which the solvability condition reads $a\mathcal{L}_{{\rm eff},\xi} f_0 = (\xi^2-\beta^{-1})b_1$ with $b_1 = \Pi P_q (\p^T \nabla_p \Phi_1)$. Therefore, $f_0(\xi) = (\beta^{-1}-\xi^2)b_1/(2\beta a)$, so that
\[
  \begin{aligned}
    \Psi_2(\q,\p,\xi) & = \Llang^{-1}\left[\left(\p^2 -\frac{n}{\beta}\right)\left(\frac{b_1}{\beta a}\Phi_{-1}-\Phi_1\right)\right] + \xi^2 \Llang^{-1}\left(\p^T \nabla_p \Phi_1-\frac{b_1}{\beta a}\p^T \nabla_p \Phi_{-1}\right) \\
    & \qquad +f_1'(\xi) \Phi_{-1}(\q,\p) + f_2(\xi).
  \end{aligned}
\]

\paragraph{Obtaining bounds on the difference of the variances}
We now choose $f_1 = f_2 = 0$ and compute
\[
  \Ladl\left(\varepsilon \Psi_{-1} + \Psi_0 + \frac{1}{\varepsilon} \Psi_1 + \frac{1}{\varepsilon^2} \Psi_2 - \phi_\varepsilon\right) = \frac{1}{\varepsilon^3} \Lnh \Psi_{2}.
\]
We deduce, in view of Corollary~\ref{col:invert}, that there exists a constant $C_\gamma \in \mathbb{R}_+$ such that, for any $\varepsilon \geq 1$,
\[
\left\|\varepsilon \Psi_{-1} + \Psi_0 + \frac{1}{\varepsilon} \Psi_1 + \frac{1}{\varepsilon^2} \Psi_2 - \phi_\varepsilon\right\|_{L^2(\pi)} \leq \frac{C_\gamma}{\varepsilon} \left\|\Lnh \Psi_{2}\right\|_{L^2(\pi)},
\]
and in fact
\[
\left\|\varepsilon \Psi_{-1} + \Psi_0 - \phi_\varepsilon\right\|_{L^2(\pi)} \leq \frac{R_{\gamma,\varphi}}{\varepsilon}
\]
for some constant $R_{\gamma,\varphi} \in \mathbb{R}_+$. The asymptotic variance ${\sigma}_{\varepsilon,\gamma}^{2}(\varphi)$ then coincides up to an error of order~$\varepsilon^{-1}$ with
\[
  \widetilde{\sigma}_{\varepsilon,\gamma}^{2}(\varphi) =   2 \int_{\xDomain}   \left ( \varepsilon \Psi_{-1} + \Psi_0\right )\Pi_{0} \varphi \, \dd \pi = 2 \int_{\xDomain} \left(\Phi_0 - \frac{b_0}{a\beta} \Phi_{-1}\right)\Pi_{0} \varphi \, \dd \pi,
\]
where we used for the second equality the fact that the average with respect to~$\pi$ of the product of a function of~$\xi$ and~$\Pi_0\varphi$ vanishes. Finally, by integrating in~$\xi$ and expressing $a,b_0$ in terms of the generator of the Langevin dynamics, namely,
\[
  \begin{aligned}
    a & = -\int_{\mathbb{R}^{2n}} \left(\Llang \Phi_{-1}\right) \Phi_{-1} \, d\pi_q\,d\pi_p = \frac{\gamma}{\beta} \|\nabla_p \Phi_{-1}\|^2_{L^2(\pi_q\pi_p)}, \\
    b_0 & = \beta \int_{\mathbb{R}^{2n}} \left(\p^2-\frac{n}{\beta}\right)\Phi_0 \, d\pi_q\,d\pi_p = -\beta \int_{\mathbb{R}^{2n}} \left(\Llang \Phi_{-1}\right) \Phi_0 \, d\pi_q\,d\pi_p
  \end{aligned}
\]
it follows that
\[
  \widetilde{\sigma}_{\varepsilon,\gamma}^{2}(\varphi) =  2 \left(\int_{\mathbb{R}^{2n}} \Phi_0 \Pi_{0} \varphi \, \dd \pi_q \, \dd \pi_q - \frac{\beta \langle \Llang \Phi_{-1},\Phi_0\rangle_{L^2(\pi_q\pi_p)}\langle \Llang \Phi_0,\Phi_{-1}\rangle_{L^2(\pi_q\pi_p)}}{\gamma \|\nabla_p \Phi_{-1}\|^2_{L^2(\pi_q\pi_p)}}\right).
\]
Now,
\[
  \begin{aligned}
    \langle \Llang \Phi_{-1},\Phi_0\rangle_{L^2(\pi_q\pi_p)} & = -\frac\gamma\beta \langle \nabla_p \Phi_{-1}, \nabla_p \Phi_0\rangle_{L^2(\pi_q\pi_p)} - \langle \Phi_{-1},\Lh \Phi_0\rangle_{L^2(\pi_q\pi_p)}, \\
    \langle \Llang \Phi_0,\Phi_{-1}\rangle_{L^2(\pi_q\pi_p)} & = -\frac\gamma\beta \langle \nabla_p \Phi_{-1}, \nabla_p \Phi_0\rangle_{L^2(\pi_q\pi_p)} + \langle \Phi_{-1}, \Lh \Phi_0\rangle_{L^2(\pi_q\pi_p)}, \\
  \end{aligned}
  \]
  so that
  \[
    \widetilde{\sigma}_{\varepsilon,\gamma}^{2}(\varphi) = \frac{2}{\beta}\left(\gamma\|\nabla_p \Phi_0\|^2_{L^2(\pi_q\pi_p)} - \gamma\frac{\langle \nabla_p \Phi_{-1}, \nabla_p \Phi_0\rangle_{L^2(\pi_q\pi_p)}^2}{\|\nabla_p \Phi_{-1}\|^2_{L^2(\pi_q\pi_p)}} + \frac{ \beta^2 \langle \Phi_{-1},\Lh \Phi_0\rangle_{L^2(\pi_q\pi_p)}^2}{\gamma \|\nabla_p \Phi_{-1}\|^2_{L^2(\pi_q\pi_p)}}\right),
  \]
which gives the claimed result.

\section{Numerical results}\label{sec:numerical}

In this section, we present the results of several numerical experiments.  First, we consider a simple illustration to demonstrate the scaling of the spectral gap as a function of~$\gamma$ and~$\varepsilon$ as predicted in Section~\ref{sec:main}. Second, we demonstrate the scaling of the asymptotic variance, as predicted in Section~\ref{sec:CLT}.  We also verify the existence of an asymptotic central limit theorem for the case of a Bayesian data analysis problem.

\subsection{Spectral gap  in Galerkin subspace}
\label{sec:spectral_gap_Galerkin}

Let $U : \RR \rightarrow \RR, \,U(\q) = \frac{1}{2}\q^{2}$. 
Moreover, denote by $\hermite_{l}$ the $l$-th Hermite polynomial as defined in \cref{eq:def:hermite}, 
and consider for prescribed integers $L \in \NN$ the finite dimensional Galerkin subspace $\mathcal{G}_{L}$ spanned by polynomials of the form 
\begin{equation}\label{eq:def:hermite:2}
\psi_{k,l,m}(\p,\xi,\q) = \hermite_{k}(\p) \hermite_{l}(\xi)\hermite_{m}(\q), ~~0 \leq l,k,m \leq L-1,
\end{equation}
and the associated projection operator 
\begin{equation}
 \Pi_{\rm Galerkin}^L  : L^{2}(\pi) \rightarrow \mathcal{G}^{L},~~\varphi ~ \mapsto ~ \sum_{k=0}^{L-1} \sum_{l=0}^{L-1} \sum_{m=0}^{L-1} u_{k,l,m}\psi_{k,l,m},
\end{equation}
where $u_{k,l,m} :=  \innerLpi{\varphi,\psi_{k,l,m}}$. In order to simplify notation we consider a linear indexing of the coefficients $u_{k,l,m}$ and the polynomials $\psi_{k,l,m}$ using a hash map of the form  $I : (k,l,m) \mapsto 1+ m + L k + L^{2} l$ so that we can write the action of the Galerkin operator on functions $\varphi \in L^{2}_{0}(\pi)$ in the compact form
\[
\Pi_{\rm Galerkin}^L \varphi = {\bm u} \cdot {\bm \psi} ,
\]
where ${\bm u} = [\widetilde{u}_{i}]_{1\leq i \leq L^{3}}$ and ${\bm \psi} = [\widetilde{\psi}_{i}]_{1\leq i \leq L^{3}}$, where $\widetilde{u}$ and $\widetilde{\psi}$ are such that $\widetilde{u}_{i} = u_{k,l,m}$ and $\widetilde{\psi}_{i} = \psi_{k,l,m}$ for $i=I(k,l,m)$.

Let $\mathcal{G}^{L}_{0} := \mathcal{G}^{L} \cap L^{2}_{0}(\pi)$. For observables $\varphi \in \mathcal{G}^{L}_{0}$, one can derive (see~\cref{sec:ap:stiffness:matrix}) a stiffness matrix ${\bm A} \in \RR^{L^{3}\times L^{3}}$ in terms of which the action of the generator $\Ladl = \Lh + \gamma \Lo+\varepsilon^{-1}\Lnh$ can  be written as
\begin{equation}
\Ladl\varphi= \Ladl \left ( {\bm u} \cdot {\bm \psi} \right ) = \left ( {\bm A}  {\bm u} \right ) \cdot {\bm \psi}.
\end{equation}
Consequently, the spectrum of $\Ladl$ in the respective Galerkin subspace is exactly given by 
the eigenvalues of ${\bm A}$ and we can numerically compute the spectral gap $  \widehat{\lambda}
_{\varepsilon,\gamma}$ of $-\Ladl$ restrained to the respective Galerkin subspace by 
diagonalizing  the matrix ${\bm A}$. Figure~\ref{fig:sg:hermite} shows the spectral gap of $-{\bm A}$ for $L=10$. As 
suggested by~\eqref{eq:spectral:gap} we observe for all considered values of $\gamma$ a scaling of $  \widehat{\lambda}_{\varepsilon,
\gamma}$ as ${\rm O}(\varepsilon^{2})$ when $\varepsilon \rightarrow 0$ and as ${\rm O}(\varepsilon^{-2})$ when $\varepsilon 
\rightarrow \infty$ (see Figure~\ref{fig:sg:hermite}, Panel A). Similarly, for fixed values of $\varepsilon$ we observe a 
scaling of $  \widehat{\lambda}_{\varepsilon,\gamma}$ as ${\rm O}(\gamma)$ when $\gamma \rightarrow 0$ and as 
${\rm O}(\gamma^{-1})$ when $\gamma \rightarrow \infty$ (see Figure~\ref{fig:sg:hermite}, Panel B). 
Finally, consider the scaling of the spectral gap $ \widehat{\lambda}_{\alpha,\alpha}$ as a function of the single scalar $\alpha$. As $\alpha \rightarrow \infty$, we expect $ \widehat{\lambda}_{\alpha,\alpha} = {\rm O}(\alpha^{-3})$, and as $\alpha \rightarrow 0$, we expect $ \widehat{\lambda}_{\alpha,\alpha} =  {\rm O}(\alpha^{3})$. Indeed, this is what we observe (see Figure~\ref{fig:sg:hermite}, Panel C).

\begin{figure}[ht]
\hspace{-.5cm}
\begin{center}
\includegraphics[width=.9\textwidth]{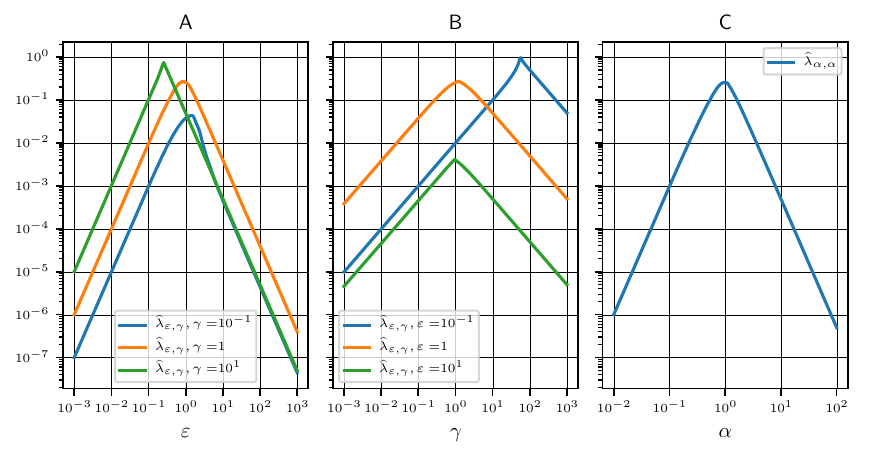}
\caption{Spectral gap, $\widehat{\lambda}
_{\varepsilon,\gamma}$, of $-\Ladl$ when considered as an operator on $\mathcal{G}^{L}_{0}$, with $L=10$. Panel~A shows $  \widehat{\lambda}
_{\varepsilon,\gamma}$ as a function of $\varepsilon$ for fixed $\gamma$. Panel~B shows $  \widehat{\lambda}
_{\varepsilon,\gamma}$ as a function of $\gamma$ for fixed $\varepsilon$. Panel~C shows $  \widehat{\lambda}
_{\alpha,\alpha}$ as a function of the scalar $\alpha$. }\label{fig:sg:hermite}
\end{center}
\end{figure}

\subsection{Scaling of asymptotic variance and demonstration of CLT}
\label{sec:num:asvar:scaling}

We next consider a simple skewed double-well potential $U : \RR \rightarrow \RR$, of the form
\begin{equation}\label{eq:def:dw}
U(\q) = \frac{b}{a} \left (\q^{2}-a\right)^{2}+ c\, \q
\end{equation}
which we parameterize as $b=1, a=1, c=1/2$. We use the BADODAB symmetric splitting scheme from~\cite{Leimkuhler2015} (see also~\cref{sec:num:int}) to simulate trajectories of the SDE \eqref{eq:adL}.
In a first set of simulations we consider different parameterizations with $\varepsilon$ taking values within the interval $[10^{-2},10]$ and $\gamma$ taking values within the interval $[10^{-4},10^{2}]$. For each parameterization we simulate $N=10,000$ independent replicas for $K=100,000$ time steps at unit temperature using a stepsize $\deltat=2\times 10^{-3}$. We randomly initialized each replica according to the associated equilibrium measure $\pi$ using a simple rejection sampling algorithm. We denote by 
\[
\widehat{\varphi}_{K} = \frac{1}{K}\sum_{k=0}^{K-1}\varphi \big( \q^{(k)}, \p^{(k)},\xi^{(k)} \big),
\]
the time average of the  observable $\varphi$  evaluated along a finite trajectory $(\q^{(k)}, \p^{(k)},\xi^{(k)})_{1\leq k \leq K}$ of the discretized process which we use as a (biased, due to discretization) Monte Carlo estimate of the expectation $\EE_{\pi}(\varphi)$. Let $\widehat{\varphi}_{K}^{(n)}$  denote the Monte Carlo estimate obtained from the trajectory of the $n$-th replica, and denote by
\begin{equation}\label{eq:empirical:mean}
\overline{\varphi}_{K}:= \frac{1}{N}\sum_{n=1}^{N} \widehat{\varphi}_{K}^{(n)},
\end{equation}
the empirical mean of the respective estimates over the $N$ independent replicas. We estimate the asymptotic variance of $\varphi$ under the discretized dynamics using
\[
\widehat{\sigma}^{2}_{\varepsilon,\gamma}(K) = \frac{1}{N}\sum_{n=0}^{N-1} \left (\widehat{\varphi}_{K}^{(n)}-\overline{\varphi}_{K}\right )^{2}.
\]
Figure~\ref{fig:fw:asymptotic:variance} shows such computed estimates of the asymptotic variance as a function of $\varepsilon$ (Panel A), and as a function of $\gamma$ (Panel B), respectively. We confirm the qualitative behaviour predicted in Section~\ref{sec:CLT} for the asymptotic variance: for fixed $\gamma=1$, the asymptotic variance $\widehat{\sigma}^{2}_{\varepsilon,\gamma}(K)$ of observables scales at most quadratically in $\varepsilon$ as $\varepsilon\rightarrow \infty$. Similarly, as $\varepsilon\rightarrow 0$, the estimated asymptotic variance $\widehat{\sigma}^{2}_{\varepsilon,\gamma}(K)$ of the observables we consider remains of order~1 (while it could increase as $\varepsilon^{-2}$ at most according to~\eqref{eq:scaling_asym_var}). For fixed $\varepsilon=1$, the estimated asymptotic variance of observables scales as at most linearly in~$\gamma$ as $\gamma\rightarrow \infty$. For the considered model system and observables the increase of the estimated asymptotic variance $\widehat{\sigma}^{2}_{\varepsilon,\gamma}(K)$ is sub-linear in $\gamma^{-1}$ as $\gamma \rightarrow 0$. We provide additional results for a slightly modified version of the model system considered here in \cref{sec:add:numexp}, where the increase of the asymptotic variance of certain observables is indeed  observed to be asymptotically linear in $\gamma^{-1}$ as $\gamma\rightarrow 0$.
\begin{figure}[ht]
  \hspace{-.5cm}
  \begin{center}
\includegraphics[width=.9\textwidth]{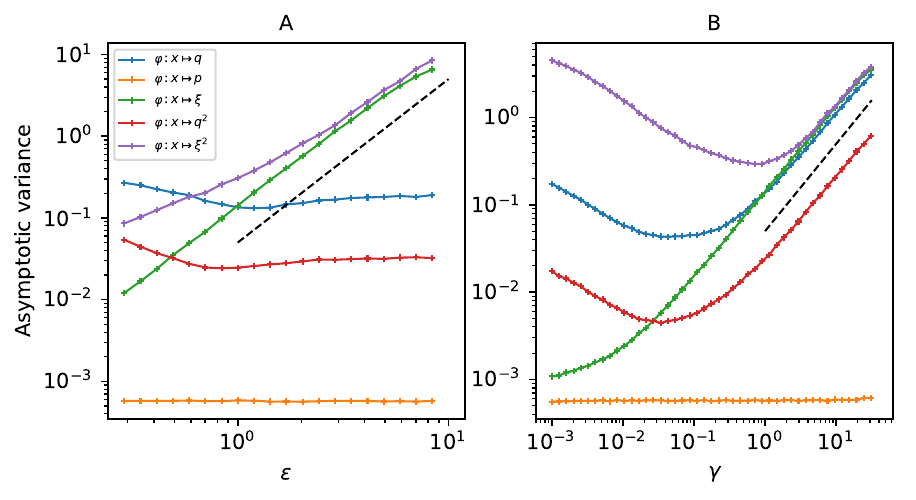}
\end{center}
\caption{Estimated asymptotic variance $\widehat{\sigma}^{2}_{\varepsilon,\gamma}(K)$ for various observables with fixed $\gamma=1$ as a function of $\varepsilon$ (Panel A), and with fixed $\varepsilon=1$ as a function of $\gamma$ (Panel B), respectively. The dashed line in Panel A corresponds to the slope of a quadratic function in $\varepsilon$. The dashed line in Panel B indicates the slope of a linear function in $\gamma$.}\label{fig:fw:asymptotic:variance}
\end{figure}

We use a second set of simulations to demonstrate the central limit theorem obtained in \cref{col:clt} for estimates $\widehat{\varphi}_{K}$ obtained as Monte Carlo estimates from the discretization of the SDE~\eqref{eq:adL}. That is, we show that for sufficiently large $K\in \NN$ the law of the estimated rescaled residual error 
\begin{equation}\label{eq:residual:error}
 \sqrt{ \frac{K\deltat}{\sigma_{\varepsilon,\gamma}^{2}(\varphi)}} (\widehat{\varphi}_{K}-\EE_{\pi}(\varphi)),
\end{equation}
 is approximately Gaussian with vanishing mean and variance $\widehat{\sigma}^{2}_{\varepsilon,\gamma}(\varphi)$ (we treat any systematic bias induced by the discretization as negligible). For parameter values $\gamma=\varepsilon=1$, we simulate $N=500,000$  independent trajectories for up to $K_{\max}=1000$ steps using the stepsize $\deltat = 10^{-1}$. For each trajecotry we compute an estimate of the rescaled residual errors by replacing $\EE_{\pi}(\varphi)$ and $\sigma_{\varepsilon,\gamma}^{2}(\varphi)$ in the expression \eqref{eq:residual:error} by the Monte Carlo estimates $\overline{\varphi}_{K_{\max}}$ and $\widehat{\sigma}^{2}_{\varepsilon,\gamma}(K_{\max})$, respectively.  Figure~\ref{fig:dw:clt}, Panel A and Figure~\ref{fig:dw:clt}, Panel B, show the empirical probability density function of the rescaled residual errors of the estimated mean and the estimated variance of the position variable $\q$, respectively. The empirical probability density functions are plotted for different values of $K$. As $K$ increases we observe that for sufficiently large $K$ the computed empirical probability density functions indeed closely follow the predicted Gaussian limiting distributions.
\begin{figure}[ht]
  \hspace{-.5cm}
  \begin{center}
\includegraphics[width=.9\textwidth]{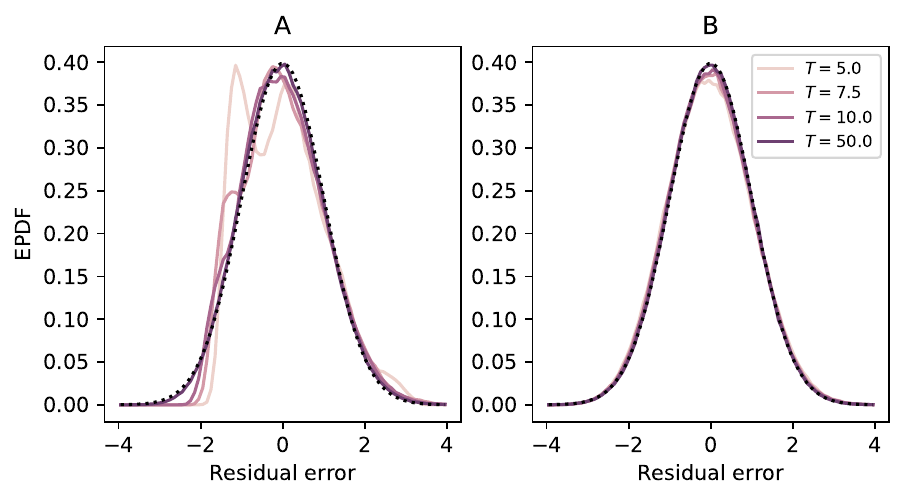}
\end{center}
\caption{Empirical probability distribution (EPDF) of the rescaled residual error at different times $T=K \deltat$. Panel A shows the EPDF of the residual error of the estimated mean of the position variable $\q$, i.e., $\varphi: (\q,\p,\xi) \mapsto \q$. Panel B shows the residual error of the estimated second moment of the position variable, i.e., $\varphi: (\q,\p,\xi) \mapsto \q^2$. Dotted lines show the density $\mathcal{N}(0,\sigma^{2}(\varphi))$, where $\sigma^{2}(\varphi)$ corresponds to the asymptotic variance of the respective observable, which is estimated using the complete trajectory data up to index $K=1000$.  }\label{fig:dw:clt}
\end{figure}

\subsection{Application to Bayesian logistic regression}\label{sec:data}
For the purpose of demonstrating the CLT in a Bayesian posterior sampling application we consider a Bayesian logistic regression trained on a subset of the MNIST benchmark data set \cite{lecun-mnisthandwrittendigit-2010} of handwritten digits for binary classification of the digits 7 and 9. We preprocess the data by means of a principal component analysis. After centering the mean of each pixel, we retain the first 100 principal components and whiten the obtained data by normalizing the variance of the corresponding loadings. The corresponding data points are denoted by $x^j$. Pictures corresponding to the number~7 are associated with $y^j=0$, while $y^j=1$ corresponds to pictures of~9. Training is run on a subset of 12,251 data points and testing on a separate subset of 2000 data points. Assuming a weakly informative Gaussian prior distribution on the parameters $\q \in \mathbb{R}^{100}$ to sample, with density $p_{0}(\q) \propto \exp(-\q^{\trans}\q/(2\sigma^{2}))$ where $\sigma^{2}=100$, and a likelihood 
\[
 p ( y^{j}, x^{j}\given \q) = \frac{ \exp \left ( y^j (x^j)^T \q \right )}{1 + \exp \left ( (x^j)^T \q\right ) },
\]
 where  $\ndata=12251$, $y^{j} \in \{0,1\}$, $x^{j}\in \RR^{100}$, the corresponding posterior distribution is of the form
\begin{align}\label{eq:posterior}
\pi(\q )\,\dd \q & \propto p_0(\q) \prod_{j=1}^{\ndata}  p ( y^{j}, x^{j}\given \q)\,\dd \q =: \exp( -U(\q) ) \,\dd \q.
\end{align}
We use the ODABADO scheme described in Appendix~\ref{sec:num:int} in order to numerically discretize  \eqref{eq:adL:1} in combination with an unbiased estimator $-\widehat{\nabla} U(\q)$ of the gradient force which we obtain by subsampling data points as specified in \eqref{eq:unbiased:estimator} using minibatches of size $m=100$. Besides the introduced gradient noise we do not apply additional random forces, i.e., $\sigma_{{\rm A}} = 0$. 
 


\begin{figure}[ht]
\hspace{-.5cm}
\begin{center}
\includegraphics[width=1.0\textwidth]{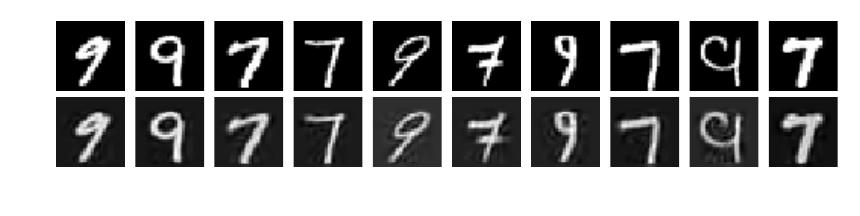}
\end{center}
\caption{Examples of images from the MNIST data set. The upper row shows the original images as obtained from the repository \cite{lecun-mnisthandwrittendigit-2010}. The lower row shows the projection of the same images onto the first 100 principal components which were used for inference in the numerical experiments presented in this article. }\label{fig:mnist:demo}
\end{figure}
In a first set of simulations we generate $N=10,000$ independent trajectories for a total number of $K=10,000$ steps using a stepsize of $\deltat=10^{-2}$ with coupling parameter $\nu=1$. We initialize the position variable of all replicas at the same location which is a point close to the mode of the target distribution, set the initial value $\xi(0)$ of the friction variable to~$0$, and for each trajectory we independently sample the initial momenta from the stationary measure, i.e., $\p(0)\sim \mathcal{N}(\0,\I_{n})$. Following the same steps as described above in the demonstration of the CLT in the previous example we compute the appropriately rescaled residual errors of the estimated mean and the estimated variance at various time points of the single coordinate variable $\q_{i}$ whose index $i=65$ we randomly selected. Figure~\ref{fig:mnist:clt} shows the histograms of the empirical distribution of the residual error of these estimates after an increasing number of time steps. Again, as in the example of the previous section we observe that for a sufficiently large number of time steps, the distribution of the residual error follows closely the anticipated Gaussian distribution. We confirm that we observe that also for other choices of the coordinate index $i$ the empirical law of the residual error converges to a centered Gaussian distribution. 
\begin{figure}[ht]
\hspace{-.5cm}
\begin{center}
  \includegraphics[width=.9\textwidth]{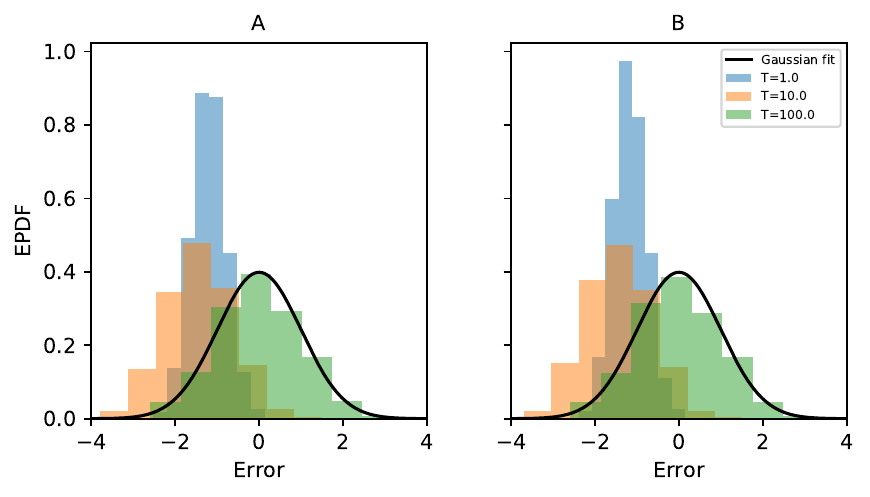}
\end{center}
\caption{Empirical probability distribution (EPDF) of the rescaled residual error at different times $T=K \deltat$ in the case of the Bayesian logistic regression posterior sampling problem. Panel A shows the EPDF of the rescaled residual error of the estimated mean of the $65$-th regression variable. Panel B shows the residual error of the estimated variance of the same regression variable}\label{fig:mnist:clt}
\end{figure}


In a second set of simulations we investigate the effect of different values of the thermal mass~$\nu$ on the convergence speed of the estimates of expectations of certain observables obtained from single trajectories. We consider the same setup as described above but generate single trajectories for different values of the coupling parameter, i.e., $\nu  = \varepsilon^2 \in \{1,10,100\}$. As observables we consider again the projection onto a single coordinate variable, and the average likelihood over the test set --a quantity commonly used for benchmarking purposes in machine learning applications, i.e., 
\[
\varphi(\q,\p,\xi) = \frac{1}{\ndatatest}\sum_{i=1}^{\ndatatest} p ( y^{j}, x^{j}\given \q) =  \frac{1}{\ndatatest}\sum_{i=1}^{\ndatatest}\frac{ \exp \left ( y^j (x^j)^T \q \right )}{1 + \exp \left ( (x^j)^T \q\right ) },
\]
where $\ndatatest=2000$, and $(x^{i},y^{i}), i =1,\dots,\ndatatest$ are the data points of the test data set.

Figure~\ref{fig:mnist:params} shows the time evolution of the corresponding Monte Carlo estimates of the mean of the 65-th component (Panel A), and the average likelihood over the test set (Panel B).  As one may have anticipated based on the asymptotic scaling of the spectral gap as ${\rm O}(\nu^{-1})$ as $\nu \rightarrow \infty$, the convergence of the respective cumulative averages (in time) of the observables under consideration becomes slower with increasing values of $\nu$. We mention that estimates appear to converge to different values in the limit $T= \deltat K\rightarrow \infty$. This observation can be explained by the fact that the invariant measure of the discretized dynamics can be expected to depend on the value of the coupling parameter $\nu$. We refer to \cite{Leimkuhler2015} for a detailed analysis of this dependency in the case of the similar BADODAB splitting scheme. A reduction of this discrepancy can be achieved by a reduction of the stepsize $\deltat$.
\begin{figure}[ht]
  \hspace{-.5cm}
  \begin{center}
\includegraphics[width=.9\textwidth]{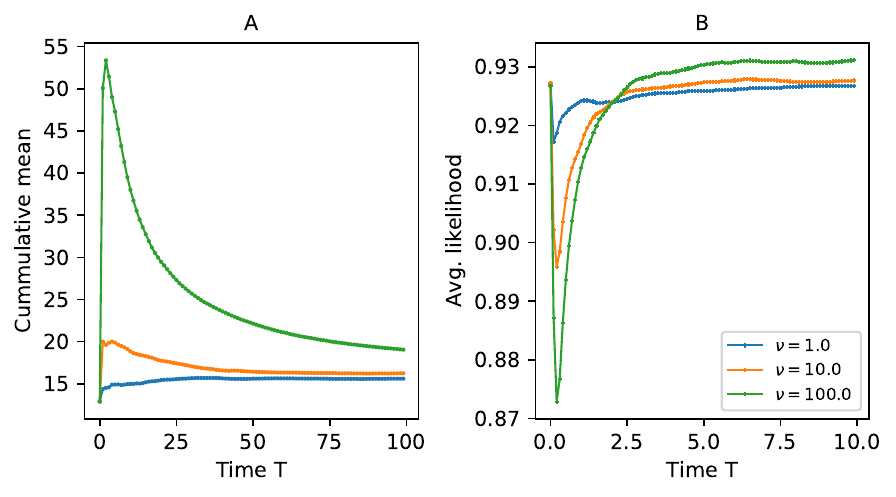}
\end{center}
\caption{Number of time steps $K$ vs. value of the Monte Carlo estimate of the mean of the $65$-th regression variable (Panel~A), and the value of the estimated average likelihood over the test (Panel~B).}\label{fig:mnist:params}
\end{figure}

\FloatBarrier

\



\section*{Acknowledgments}
The authors thank Laurent Michel and Lois Delande for pointing out a mistake in an earlier version of Lemma~\ref{lem:bound_AA_Ace}. This work was initiated during the authors' stay at the Institut Henri Poincar\'e - Centre Emile Borel during the trimester ``Stochastic Dynamics Out of Equilibrium'' (April-July 2017). The authors warmly thank this institution for its hospitality. The research of B. Leimkuhler was supported by the ERC project RULE (grant number 320823) and  EPSRC grant EP/P006175/1. The work of M. Sachs was supported by the National Science Foundation under grant DMS-1638521 to the Statistical and Applied Mathematical Sciences Institute (SAMSI), North Carolina. The activity of Gabriel Stoltz was funded in part by the Agence Nationale de la Recherche, under grant ANR-14-CE23-0012 (COSMOS), and by the European Research Council under the European Union’s Seventh Framework Programme (FP/2007-2013)/ERC Grant Agreement number 614492. G.S. also benefited from the scientific environment of the Laboratoire International Associ\'e between the Centre National de la Recherche Scientifique and the University of Illinois at Urbana-Champaign.

\bibliographystyle{siamplain}
\bibliography{matthias_refs.bib}

\appendix
\newpage
\section*{Appendix}
This appendix contains additional details on the numerical experiments presented in Section \ref{sec:numerical} of the main text. In Section \ref{sec:ap:stiffness:matrix} we present details on the derivation of the stiffness matrix of the generator of adaptive Langevin dynamics in the considered Galerkin subspace. Section \ref{sec:add:numexp} contains complementary numerical experiments which demonstrate the predicted asymptotic scaling of the asymptotic variance of some observables as $1/\gamma$ in the underdamped limit. Section \ref{sec:num:int}  details the numerical integrators used to obtain the results presented in Sections \ref{sec:num:asvar:scaling} and \ref{sec:data} of the main text.

\section{Derivation of the stiffness matrix in the Hermite Galerkin projection}\label{sec:ap:stiffness:matrix}
In this section we outline the derivation of the stiffness matrix ${\bm A} = \gamma {\bm A}_{\rm OU}+\varepsilon^{-1}{\bm A}_{{\rm NH}} + {\bm A}_{{\rm H}}$, where $ {\bm A}_{\rm OU}$, ${\bm A}_{{\rm NH}}$, and ${\bm A}_{{\rm H}}$ denote the stiffness matrices associated with the generators $\Lo, \,\Lnh$, and $\Lh$, respectively.
Let $\hermite_{l}$ denote the $l$-th Hermite Polynomial, i.e., 
\begin{equation}\label{eq:def:hermite}
\hermite_{l}(x) = \frac{1}{\sqrt{l!}} \widetilde{H}_{l} \left ( \sqrt{\beta} x \right ),~ \widetilde{H}_{l}(x) = (-1)^{l} {\rm e}^{x^{2}/2} \frac{\dd^{l}}{\dd x^{l}} \left ( {\rm e}^{-x^{2}/2} \right ).
\end{equation}
Simple computations show 
\begin{equation}\label{eq:diff:action:hermite}
\partial_{x} \hermite_{l}(x)  =  \sqrt{\beta l} \hermite_{l-1}(x), ~~ \partial_{x}^{*} \hermite_{l}(x)  = \sqrt{\beta (l+1)} \hermite_{l+1}(x),
\end{equation}
where $\partial_{x}^{*}$ denotes the adjoint of $\partial_{x}$ in $L^{2}({\rm e}^{-(\beta/2) x^{2}}\dd x)$.
Rewriting the generators $\Lo, \Lnh$ and $\Lh$ in terms of the operators $\partial_{p},\partial_{p}^{*},\partial_{\xi},\partial_{\xi}^{*},\partial_{q} $ and $\partial_{q}^{*}$ (see~\eqref{eq:def_Lh_Lo} and~\eqref{eq:def_Lnh_rescaled}), and using \eqref{eq:diff:action:hermite} 
we find
\begin{align}\label{eq:Lo:action}
\Lo \psi_{k,l,m} &= -k\psi_{k,l,m},\\
\begin{split}
\Lnh \psi_{k,l,m} 
&=\beta^{-1/2} \Big (k \sqrt{l} \psi_{k,l-1,m}  + \sqrt{(k+1)(k+2)l} \psi_{k+2,l-1,m} \\ 
&\quad \qquad \qquad - k \sqrt{l+1} \psi_{k,l+1,m} - \sqrt{k(k-1)(l+1)} \psi_{k-2,l+1,m} \Big ),\label{eq:Lnh:action}
\end{split}\\
\Lh \psi_{k,l,m} &= \sqrt{m(k+1)}\psi_{k+1,l,m-1}-\sqrt{(m+1)k}\psi_{k-1,l,m+1},\label{eq:Lh:action}
\end{align}
with $\psi_{k,l,m}$ as defined in \eqref{eq:def:hermite:2}. For $i,j\in \NN$, let $E_{i,j}\in \RR^{L^{3}\times L^{3}}$ denote the matrix with entries 
\begin{equation}
E_{i,j} :=
\begin{cases}
[\delta_{i,i^{\prime}}\delta_{j,j^{\prime}}]_{1\leq i^{\prime},j^{\prime} \leq L^{3}}, &\text{ if } 1\leq i,j \leq L^{3}\\
0,& \text{ otherwise}.
\end{cases}
\end{equation}
Then, recalling the definition of the hash function~$I$ given in Section~\ref{sec:spectral_gap_Galerkin} and defining $\tilde{I} : (k,l,m) \mapsto I(k,l,m) \mathbbm{1}_{\{0,1,\dots,L-1\}}(k)\mathbbm{1}_{\{0,1,\dots,L-1\}}(l)\mathbbm{1}_{\{0,1,\dots,L-1\}}(m) $, the stiffness matrices associated with the operators $\Lo,\Lnh$, and $\Lh$ follow from \eqref{eq:Lo:action}--\eqref{eq:Lh:action} as
\begin{align*}
{\bm A}_{\rm OU} &= \sum_{k=0}^{L-1}\sum_{l=0}^{L-1}\sum_{m=0}^{L-1} -k E_{\tilde{I}(k,l,m),\tilde{I}(k,l,m)},\\
\begin{split}
 {\bm A}_{\rm NH} &=  \beta^{-1/2} \sum_{k=0}^{L-1}\sum_{l=0}^{L-1} \sum_{m=0}^{L-1}\Big ( -  k\sqrt{l}E_{\tilde{I}(k,l-1,m),\tilde{I}(k,l,m)} - \sqrt{(k+1)(k+2) l}E_{\tilde{I}(k+2,l-1,m),\tilde{I}(k,l,m)}\\
 &\qquad \qquad\qquad- k\sqrt{l+1}E_{\tilde{I}(k,l+1,m),\tilde{I}(k,l,m)} + \sqrt{k(k-1)(l+1)} E_{\tilde{I}(k-2,l+1,m),\tilde{I}(k,l,m)} \Big),
\end{split}\\
{\bm A}_{\rm H}  &= \sum_{k=0}^{L-1}\sum_{l=0}^{L-1}\sum_{m=0}^{L-1}  \left (\sqrt{m(k+1)} E_{\tilde{I}(k+1,l,m-1),\tilde{I}(k,l,m)} -\sqrt{(m+1)k}E_{\tilde{I}(k-1,l,m+1),\tilde{I}(k,l,m)}\right),
\end{align*}
respectively.

\section{Additional numerical experiment}\label{sec:add:numexp}
In order to demonstrate the predicted behaviour of the asymptotic variance as $\gamma\rightarrow 0$, we consider the setup described in Section~\ref{sec:num:asvar:scaling} with the modified parametrization $a=1,\; b=4,\; c=1/2$ of the potential function $U(\q) = \frac{b}{a} \left (\q^{2}-a\right)^{2}+ c\, \q$. This change of parameterization results in an increased barrier height between the two local minima of the potential function. Using a stepsize of $\deltat=10^{-1}$, we obtain estimates of the asymptotic variance for certain observables by following the same procedure and using the same number of replicas and iterations as in the numerical experiment described in Section~\ref{sec:num:asvar:scaling}. Fig~\ref{fig:as:var:2} shows the estimated asymptotic variance as a function of the friction coefficient $\gamma$. The value of the coupling parameter, $\varepsilon$, was set to $1$ in all runs. As predicted, we find that for certain observables the asymptotic variance increases  linearly in $\gamma^{-1}$ as $\gamma\rightarrow 0$.

\begin{figure}[ht]
\begin{centering}
\hspace{-.5cm}
\includegraphics[width=.45\textwidth]{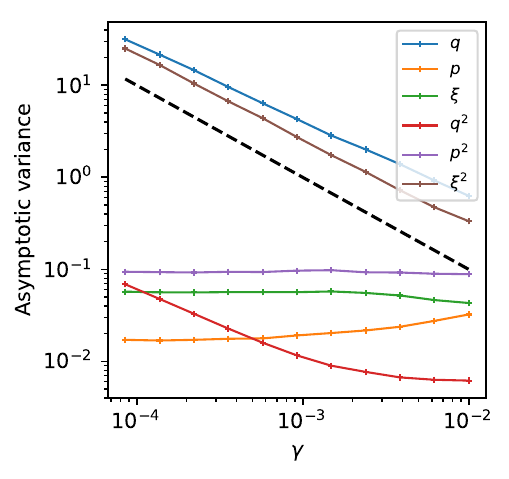}
\caption{Estimated asymptotic variance $\widehat{\sigma}^{2}_{\varepsilon,\gamma}(K)$ for various observables with fixed $\varepsilon=1$  as a function of $\gamma$. The dashed line corresponds to the slope of a linear function in $\gamma^{-1}$.}\label{fig:as:var:2}
\end{centering}
\end{figure}

\section{Numerical integrators}\label{sec:num:int}
In this section we briefly describe the construction of the numerical integrators for the SDEs \eqref{eq:adL} and \eqref{eq:adL:1}, respectively, which we use in the numerical experiments described in Section~\ref{sec:numerical}. We construct these integrators as Strang splittings based the decomposition of the generator into elementary pieces. 
\subsection{Numerical integrator for \eqref{eq:adL}} 
Denote by
\[
\La = \p \cdot \nabla_{\q},~~ \Lb=-\nabla U(\q) \cdot \nabla_{\p}, ~~\Ld = \frac{1}{\varepsilon}\left ( |\p|^2 - \frac{n}{\beta} \right ) \partial_{\xi}
\] 
the Liouville operators associated with the differential equations
\begin{equation}
\dot{\q} =\p,~~\dot{\p} = -\nabla U(\q),~ ~\dot{\xi} = \frac{1}{\varepsilon}\left ( |\p|^2 - \frac{n}{\beta} \right )
\end{equation}
respectively. Moreover, denote by $\widetilde{\Lo} = - \left (\gamma+\frac{\xi}{\varepsilon} \right ) \p \cdot \nabla_{\p} + \frac{\gamma}{\beta} \Delta_{\p}$ the generator associated with the differential equation
\begin{equation}
\dot{\p} = - \left (\gamma+\frac{\xi}{\varepsilon} \right) \p + \sqrt{\frac{2\gamma}{\beta}} \dot{\W}.
\end{equation}
Define coefficients  
\begin{equation}\label{eq:coeff}
\alpha(\zeta,\deltat) :=  {\rm e}^{-\deltat \zeta} ,~~G(\sigma,\zeta, \deltat )  := 
\begin{cases}
\sigma \sqrt{(1-{\rm e}^{-2\deltat\zeta})/(2\zeta)},& \text{ if } \zeta \neq 0,\\
\sigma\sqrt{\deltat},& \text{ if} \zeta = 0, 
\end{cases}
\end{equation}
so that the stochastic update
\begin{equation}
\p_{k+1} = \alpha(\zeta,\deltat)\p_{k} + G(\sigma,\zeta, \deltat ) \mathcal{R}_{k}, \quad \mathcal{R}_{k} \sim \mathcal{N}(\0,\I_{n}),
\end{equation}
is equivalent to evolving the SDE $\dot{\p} = - \zeta \p + \sigma\dot{\W}, ~\zeta,\sigma \in \RR$ for time $\deltat \geq 0$.\\

Consider the numerical method
\begin{align*}
\p_{k+1/2} &= \p_{k} - \frac{\deltat}{2}\nabla U(\q_{k}),\\
\q_{k+1/2} &= \q_{k} +\frac{\deltat}{2}\p_{k+1/2},\\
\xi_{k+1/2} &= \xi_{k} +\frac{\deltat}{2\varepsilon}\left ( |\p_{k+1/2}|^2 - \frac{n}{\beta} \right ),\\
 \hat{\p}_{k+1/2}&=  \alpha \left (\varepsilon^{-1}\xi_{k+1/2} + \gamma,\deltat \right )\p_{k+1/2} +
 G \left ( \sqrt{\frac{2\gamma}{\beta}} ,\varepsilon^{-1}\xi_{k+1/2} + \gamma, \deltat \right )\mathcal{R}_{k}, ~  \mathcal{R}_{k} \sim \mathcal{N}(\0,\I_{n}),\\
\xi_{k+1} &= \xi_{k+1/2} +\frac{\deltat}{2\varepsilon}\left ( |\hat{\p}_{k+1/2}|^2 - \frac{n}{\beta} \right ),\\
\q_{k+1} &= \q_{k+1/2} +\frac{\deltat}{2}\hat{\p}_{k+1/2},\\
\p_{k+1} &= \hat{\p}_{k+1/2} - \frac{\deltat}{2}\nabla U(\q_{k+1}).\\
\end{align*}
This corresponding to a symmetric splitting the propagator of the SDE \eqref{eq:adL}:
\begin{equation}
{\rm e}^{\deltat \Ladl} = 
{\rm e}^{\frac{\deltat}{2} \Lb } 
{\rm e}^{\frac{\deltat}{2} \La }
{\rm e}^{\frac{\deltat}{2} \Ld }
{\rm e}^{\deltat \widetilde{\Lo} }
{\rm e}^{\frac{\deltat}{2} \Ld }
{\rm e}^{\frac{\deltat}{2} \La }
{\rm e}^{\frac{\deltat}{2} \Lb } 
+ {\rm O}(\deltat^{3}),
\end{equation}
in accordance with the naming in \cite{Leimkuhler2015}. We refer to this as the BADODAB scheme.

\subsection{Numerical integrator for \eqref{eq:adL:1}}
While the above BADODAB integration scheme can be adapted to the setup of \eqref{eq:adL:1}, the resulting numerical scheme does not correspond to a splitting of propagator of the respective SDE in the presence of a gradient noise. In particular the weak convergence order of that integrator can only be expected to be of order 1 (this is in comparison to a weak convergence of order 2 in the absence of a gradient noise). In what follows we briefly describe an integrator for the SDE  \eqref{eq:adL:1} which in the presence of a gradient noise still corresponds to a symmetric splitting of the associated propagator, which means that the weak error as well as the error in ergodic averages decay at least quadratically as $\deltat \rightarrow 0$. 

Let the operator $\La$ be as defined above. Denote by  
\begin{equation}
\widetilde{\Ld} =  \frac{1}{\nu}\left ( |\p|^2 - \frac{n}{\beta} \right ) \partial_{\zeta}
\end{equation}
the Liouville operator associated with the differential equation 
\begin{equation}
\dd \zeta = \frac{1}{\nu}\left ( |\p|^2- \frac{n}{\beta}  \right ) \dd t,
\end{equation}
and denote by
\begin{equation}
\widetilde{\Lo} = -\zeta \p \cdot \nabla_{p} + \frac{\sigma_{A}^{2}}{2} \Delta_{p},
\end{equation}
the generator of the SDE 
\begin{equation}
\dd \p = - \zeta \p \dd t +  \sigma_{\rm A} \, \dd \W_{\rm A}.
\end{equation}

If the exact gradient force is replaced by an unbiased estimator $-\widehat{\nabla} U(\q_{k})$, then, under the assumption that the residual error $\mathcal{R}_{{\rm G},k} =\nabla U(\q_{k}) -\widehat{\nabla} U(\q_{k})$ is Gaussian, and independent of the value of $\q_{k}$, i.e., $\mathcal{R}_{{\rm G},k} \sim \mathcal{N}(\0, \widetilde{\sigma}_{\rm G}^{2})$, where $\widetilde{\sigma}_{\rm G}^{2} = \textrm{var}(  \widehat{\nabla} U(\q))$, an Euler update of the form $\p_{k+1} = \p_{k} - \deltat \widehat{\nabla} U(\q_{k})$ can be viewed as an exact solution of the SDE
\begin{equation}
\dd \p =  -\nabla U(\q) \dd t  + \sqrt{\deltat}\, \widetilde{\sigma}_{\rm G} \, \dd \W_{\rm G},
\end{equation}
with associated generator 
\begin{equation}
\widetilde{\Lb} = -\nabla U(\q) \cdot \nabla_p + \deltat \frac{\sigma_{G}^{2}}{2} \Delta_{p}.
\end{equation}
Let 
\begin{align*}
\p_{k+1/2}&= \alpha \left (\zeta_{k},\deltat/2 \right ) \p_{k} + G\left(\sigma_{A},\zeta_{k}, \deltat/2 \right)\mathcal{R}_{k}, \quad \mathcal{R}_{k} \sim \mathcal{N}(\0,\I_{n}),\\
\zeta_{k+1/2} &= \zeta_{k} +\frac{\deltat}{2\nu}\left ( |\p_{k+1/2}|^2 - \frac{n}{\beta} \right ),\\
\q_{k+1/2} &= \q_{k} +\frac{\deltat}{2}\M^{-1}\p_{k+1/2},\\
\hat{\p}_{k+1/2} &= \p_{k+1/2} - \deltat  \widehat{\nabla} U(\q_{k+1/2}),\\
\q_{k+1} &= \q_{k+1/2} +\frac{\deltat}{2}\M^{-1}\hat{\p}_{k+1/2},\\
\zeta_{k+1} &= \zeta_{k+1/2} +\frac{\deltat}{2\nu}\left ( |\hat{\p}_{k+1/2}|^2 - \frac{n}{\beta} \right ),\\
\p_{k+1}&= \alpha \left (\zeta_{k+1},\deltat/2 \right ) \p_{k} + G\left(\sigma_{A},\zeta_{k+1}, \deltat/2 \right)\mathcal{R}_{k+1/2} , \quad \mathcal{R}_{k+1/2} \sim \mathcal{N}(\0,\I_{n}),\\
\end{align*}
with coefficients $\alpha,G$ as defined in \eqref{eq:coeff}.  This corresponds to the following decomposition of the propagator of the SDE \eqref{eq:adL:1}:
\begin{equation}
{\rm e}^{\deltat \Ladl} = 
{\rm e}^{\frac{\deltat}{2} \widetilde{\Lo} }
{\rm e}^{\frac{\deltat}{2} \widetilde{\Ld} }
{\rm e}^{\frac{\deltat}{2} \La }
{\rm e}^{\frac{\deltat}{2} \widetilde{\Lb} } 
{\rm e}^{\frac{\deltat}{2} \La }
{\rm e}^{\frac{\deltat}{2} \widetilde{\Ld} }
{\rm e}^{\frac{\deltat}{2} \widetilde{\Lo} }
+ {\rm O}(\deltat^{3}).
\end{equation}
We refer to this as the ODABADO scheme.
\section*{Acknowledgments} The above discussed ODABADO splitting scheme was previously proposed in 2016 as a second order scheme for noisy gradient systems by Xiaocheng Shang.

\end{document}

%% file: AdLD_shared.tex
\usepackage{lipsum}
\usepackage{amsfonts}
\usepackage{graphicx}
\usepackage{epstopdf}
\usepackage{algorithmic}
\ifpdf
  \DeclareGraphicsExtensions{.eps,.pdf,.png,.jpg}
\else
  \DeclareGraphicsExtensions{.eps}
\fi

\newsiamremark{remark}{Remark}
\newsiamremark{hypothesis}{Hypothesis}
\crefname{hypothesis}{Hypothesis}{Hypotheses}
\newsiamthm{claim}{Claim}

\headers{Hypocoercivity properties of adaptive Langevin}{B. Leimkuhler, M. Sachs, and G. Stoltz}

\title{Hypocoercivity properties of adaptive Langevin dynamics\thanks{Submitted to the editors DATE.
}}

\author{
 Benedict Leimkuhler\thanks{School of Mathematics, University of Edinburgh, UK
  (\email{b.leimkuhler@ed.ac.uk}, \url{http://kac.maths.ed.ac.uk/\string~bl/}).}
\and Matthias Sachs \thanks{Department of Mathematics, Duke University, USA
  (\email{msachs@math.duke.edu}, \url{https://math.duke.edu/people/matthias-ernst-sachs}).}
\and  Gabriel Stoltz\thanks{Universit\'e Paris-Est, CERMICS (ENPC),
  Inria, F-77455 Marne-la-Vall\'ee, France
(\email{gabriel.stoltz@enpc.fr}, \url{http://cermics.enpc.fr/\~stoltz})}
}

\usepackage{amsopn}


%% file: AdLD_macros.tex
\usepackage{amssymb}
\usepackage{amsmath}
\usepackage{geometry}                
\usepackage{placeins}
\usepackage{graphicx}
\usepackage{color}
\usepackage{multirow}
\usepackage{tablefootnote}

\usepackage{mathtools,xparse}
\usepackage{mathrsfs}
\usepackage{enumitem} 
\usepackage{placeins}
\usepackage{accents}
\usepackage{bm}

\def\dd{{\rm d}}

\def\ndatatest{\hat{N}}

\def\ndata{\widetilde{N}}

\def\hermite{h}
\def\Lfd{\mathcal{G}_{\varepsilon} } 
\def\AA{A_{\varepsilon}}
\def\Ace{\mathscr{A}_{\varepsilon}}
\def\Ladl{\Lc_{\rm AdL}}
\def\Lo{\Lc_{\rm O}}
\def\Lnh{\Lc_{\rm NH}}
\def\Lh{\Lc_{\rm H}}
\def\La{\Lc_{\rm A}}
\def\Lb{\Lc_{\rm B}}
\def\Ld{\Lc_{\rm D}}

\def\Ac{\mathcal{A}}

\def\qpoincare{\kappa_{q}}

\def\aa{{a_{\varepsilon,\gamma}}}
\def\const{{C}}

\def\NN{{\mathbb{N}}}

\def\trans{{T}}

\def\xDomain{{\RR^{2n+1}}}

\def\EE{{ \mathbb{E}}}
\def\RR{{\mathbb R}}

\def\FF{{\bf F}}

\def\deltat{{\Delta t }}

\makeatletter
\DeclareFontFamily{OMX}{MnSymbolE}{}
\DeclareSymbolFont{MnLargeSymbols}{OMX}{MnSymbolE}{m}{n}
\SetSymbolFont{MnLargeSymbols}{bold}{OMX}{MnSymbolE}{b}{n}
\DeclareFontShape{OMX}{MnSymbolE}{m}{n}{
    <-6>  MnSymbolE5
   <6-7>  MnSymbolE6
   <7-8>  MnSymbolE7
   <8-9>  MnSymbolE8
   <9-10> MnSymbolE9
  <10-12> MnSymbolE10
  <12->   MnSymbolE12
}{}
\DeclareFontShape{OMX}{MnSymbolE}{b}{n}{
    <-6>  MnSymbolE-Bold5
   <6-7>  MnSymbolE-Bold6
   <7-8>  MnSymbolE-Bold7
   <8-9>  MnSymbolE-Bold8
   <9-10> MnSymbolE-Bold9
  <10-12> MnSymbolE-Bold10
  <12->   MnSymbolE-Bold12
}{}

\let\llangle\@undefined
\let\rrangle\@undefined
\DeclareMathDelimiter{\llangle}{\mathopen}%
                     {MnLargeSymbols}{'164}{MnLargeSymbols}{'164}
\DeclareMathDelimiter{\rrangle}{\mathclose}%
                     {MnLargeSymbols}{'171}{MnLargeSymbols}{'171}
\makeatother

\def\({\left (}
\def\){\right )}

\DeclarePairedDelimiter{\abs}{\lvert}{\rvert}
\DeclarePairedDelimiter{\norm}{\lVert}{\rVert}

\DeclarePairedDelimiter{\innerLpi}{\langle}{\rangle_{L^{2}(\pi)}}

\DeclarePairedDelimiter{\innera}{\llangle}{\rrangle_{\varepsilon,\gamma}}
\NewDocumentCommand{\normLinfty}{ s O{} m }{%
  \IfBooleanTF{#1}{\norm*{#3}}{\norm[#2]{#3}}_{L_\infty}%
}
\NewDocumentCommand{\normL}{ s O{} m }{%
  \IfBooleanTF{#1}{\norm*{#3}}{\norm[#2]{#3}}_{L^{2}(\mu)}%
}
\NewDocumentCommand{\normLmu}{ s O{} m }{%
  \IfBooleanTF{#1}{\norm*{#3}}{\norm[#2]{#3}}_{L^{2}(\mu)}%
}
\NewDocumentCommand{\normLpi}{ s O{} m }{%
  \IfBooleanTF{#1}{\norm*{#3}}{\norm[#2]{#3}}_{L^{2}(\pi)}%
}
\NewDocumentCommand{\normLtwo}{ s O{} m }{%
  \IfBooleanTF{#1}{\norm*{#3}}{\norm[#2]{#3}}_{L_2}%
}
\NewDocumentCommand{\normEuc}{ s O{} m }{%
  \IfBooleanTF{#1}{\norm*{#3}}{\norm[#2]{#3}}_{2}%
}
\NewDocumentCommand{\normFrob}{ s O{} m }{%
  \IfBooleanTF{#1}{\norm*{#3}}{\norm[#2]{#3}}_{{\rm F}}%
}

\newcommand\given[1][]{\:#1\vert\:}

\def\W{{\bf W}}

\def\I{{\bf I}}

\def\dd{{\rm d}}

\def\q{{ \bf q}}

\def\M{{ \bf M}}

\def\p{{ \bf p}}

\def\0{{ \bf 0}}

\def\Lc{\mathcal{L}}

\def\Ac{\mathcal{A}}

%% file: AdLd-main-siam-arxive-mix.bbl
\begin{thebibliography}{10}

\bibitem{bakry2008simple}
{\sc D.~Bakry, F.~Barthe, P.~Cattiaux, and A.~Guillin}, {\em A simple proof of
  the {P}oincar{\'e} inequality for a large class of probability measures},
  Electronic Communications in Probability, 13 (2008), pp.~60--66.

\bibitem{Bhattacharya1982}
{\sc R.~N. Bhattacharya}, {\em On the functional central limit theorem and the
  law of the iterated logarithm for {M}arkov processes}, Zeitschrift f{\"u}r
  Wahrscheinlichkeitstheorie und verwandte Gebiete, 60 (1982), pp.~185--201.

\bibitem{bou2010long}
{\sc N.~Bou-Rabee and H.~Owhadi}, {\em Long-run accuracy of variational
  integrators in the stochastic context}, SIAM Journal on Numerical Analysis,
  48 (2010), pp.~278--297.

\bibitem{cances2007theoretical}
{\sc E.~Cances, F.~Legoll, and G.~Stoltz}, {\em Theoretical and numerical
  comparison of some sampling methods for molecular dynamics}, ESAIM Math.
  Model. Numer. Anal., 41 (2007), pp.~351--389.

\bibitem{Chen2014}
{\sc T.~Chen, E.~Fox, and C.~Guestrin}, {\em Stochastic gradient {H}amiltonian
  {M}onte {C}arlo}, in International Conference on Machine Learning, 2014,
  pp.~1683--1691.

\bibitem{cheng2017underdamped}
{\sc X.~Cheng, N.~S. Chatterji, P.~L. Bartlett, and M.~I. Jordan}, {\em
  Underdamped {L}angevin {M}{C}{M}{C}: {A} non-asymptotic analysis}, in
  Proceedings of the 31st Conference On Learning Theory, vol.~75 of Proceedings
  of Machine Learning Research, PMLR, 06--09 Jul 2018, pp.~300--323.

\bibitem{delande}
{\sc L.~Delande}, {\em Sharp spectral gap of adaptive langevin dynamics},
  preprint,  (2023).

\bibitem{Ding2014}
{\sc N.~Ding, Y.~Fang, R.~Babbush, C.~Chen, R.~D. Skeel, and H.~Neven}, {\em
  {B}ayesian sampling using stochastic gradient thermostats}, in Advances in
  neural information processing systems, 2014, pp.~3203--3211.

\bibitem{DMS09}
{\sc J.~Dolbeault, C.~Mouhot, and C.~Schmeiser}, {\em Hypocoercivity for
  kinetic equations with linear relaxation terms}, C. R. Math. Acad. Sci.
  Paris, 347 (2009), pp.~511--516.

\bibitem{Dolbeault2015}
{\sc J.~Dolbeault, C.~Mouhot, and C.~Schmeiser}, {\em Hypocoercivity for linear
  kinetic equations conserving mass}, Transactions of the American Mathematical
  Society, 367 (2015), pp.~3807--3828.

\bibitem{eberle2011reflection}
{\sc A.~Eberle}, {\em Reflection coupling and {W}asserstein contractivity
  without convexity}, Comptes Rendus Mathematique, 349 (2011), pp.~1101--1104.

\bibitem{eberle2019}
{\sc A.~Eberle, A.~Guillin, and R.~Zimmer}, {\em Couplings and quantitative
  contraction rates for {L}angevin dynamics}, Ann. Probab., 47 (2019),
  pp.~1982--2010.

\bibitem{Hairer2011a}
{\sc M.~Hairer and J.~C. Mattingly}, {\em Yet another look at {H}arris ergodic
  theorem for {M}arkov chains}, in Seminar on Stochastic Analysis, Random
  Fields and Applications VI, vol.~63, Springer, 2011, pp.~109--117.

\bibitem{herzog2018exponential}
{\sc D.~P. Herzog}, {\em Exponential relaxation of the {N}os\'e-{H}oover
  equation under {B}rownian heating}, Communications in Mathematical Sciences,
  16 (2018), pp.~2231--2260.

\bibitem{Hor67}
{\sc L.~H{\"o}rmander}, {\em Hypoelliptic second order differential equations},
  Acta Math., 119 (1967), pp.~147--171.

\bibitem{iacobucci2017convergence}
{\sc A.~Iacobucci, S.~Olla, and G.~Stoltz}, {\em Convergence rates for
  nonequilibrium {L}angevin dynamics}, Annales math{\'e}matiques du Qu{\'e}bec,
  43 (2019), pp.~73--98.

\bibitem{Jones2011a}
{\sc A.~Jones and B.~Leimkuhler}, {\em Adaptive stochastic methods for sampling
  driven molecular systems}, The Journal of Chemical Physics, 135 (2011).

\bibitem{kliemann1987recurrence}
{\sc W.~Kliemann}, {\em Recurrence and invariant measures for degenerate
  diffusions}, Ann. Probab., 15 (1987), pp.~690--707.

\bibitem{KopecLangevin}
{\sc M.~Kopec}, {\em Weak backward error analysis for {L}angevin process}, BIT
  Numerical Mathematics, 55 (2015), pp.~1057--1103.

\bibitem{lecun-mnisthandwrittendigit-2010}
{\sc Y.~LeCun and C.~Cortes}, {\em {MNIST} handwritten digit database},
  (2010), \url{http://yann.lecun.com/exdb/mnist/}.

\bibitem{LeMaBook}
{\sc B.~Leimkuhler and C.~Matthews}, {\em Molecular Dynamics: With
  Deterministic and Stochastic Numerical Methods}, Interdisciplinary Applied
  Mathematics, Springer, 2015.

\bibitem{LeMaSt2015}
{\sc B.~Leimkuhler, C.~Matthews, and G.~Stoltz}, {\em The computation of
  averages from equilibrium and nonequilibrium {L}angevin molecular dynamics},
  IMA Journal of Numerical Analysis, 36 (2015), pp.~13--79.

\bibitem{Leimkuhler2015}
{\sc B.~Leimkuhler and X.~Shang}, {\em Adaptive thermostats for noisy gradient
  systems}, SIAM Journal on Scientific Computing, 38 (2016), pp.~A712--A736.

\bibitem{Lelievre2016a}
{\sc T.~Leli{\`e}vre and G.~Stoltz}, {\em Partial differential equations and
  stochastic methods in molecular dynamics}, Acta Numerica, 25 (2016),
  pp.~681--880.

\bibitem{Mattingly2002}
{\sc J.~C. Mattingly, A.~M. Stuart, and D.~J. Higham}, {\em Ergodicity for
  {SDEs} and approximations: locally {L}ipschitz vector fields and degenerate
  noise}, Stochastic Processes and their Applications, 101 (2002),
  pp.~185--232.

\bibitem{meyn1993stability}
{\sc S.~P. Meyn and R.~L. Tweedie}, {\em Stability of {M}arkovian processes
  {II}: Continuous-time processes and sampled chains}, Advances in Applied
  Probability, 25 (1993), pp.~487--517.

\bibitem{Meyn1997}
{\sc S.~P. Meyn and R.~L. Tweedie}, {\em Markov Chains and Stochastic
  Stability}, Springer Science \& Business Media, 2012.

\bibitem{nose1984unified}
{\sc S.~Nos{\'e}}, {\em A unified formulation of the constant temperature
  molecular dynamics methods}, The Journal of Chemical Physics, 81 (1984),
  pp.~511--519.

\bibitem{Pavliotis2014}
{\sc G.~A. Pavliotis}, {\em Stochastic Processes and Applications}, Springer,
  2016.

\bibitem{PV08}
{\sc G.~A. Pavliotis and A.~Vogiannou}, {\em Diffusive transport in periodic
  potentials: underdamped dynamics}, Fluctuation and Noise Letters, 08 (2008),
  pp.~L155--L173.

\bibitem{Rey-Bellet2006a}
{\sc L.~Rey-Bellet}, {\em Ergodic properties of {M}arkov processes}, in Open
  Quantum Systems II, S.~Attal, A.~Joye, and C.-A. Pillet, eds., vol.~1881 of
  Lecture Notes in Mathematics, Springer, 2006, pp.~1--39.

\bibitem{roussel2017spectral}
{\sc J.~Roussel and G.~Stoltz}, {\em Spectral methods for {L}angevin dynamics
  and associated error estimates}, ESAIM Math. Model. Numer. Anal., 52 (2018),
  pp.~1051--1083.

\bibitem{sachs2017langevin}
{\sc M.~Sachs, B.~Leimkuhler, and V.~Danos}, {\em {L}angevin dynamics with
  variable coefficients and nonconservative forces: From stationary states to
  numerical methods}, Entropy, 19 (2017).

\bibitem{Shang2015}
{\sc X.~Shang, Z.~Zhu, B.~Leimkuhler, and A.~J. Storkey}, {\em
  Covariance-controlled adaptive {L}angevin thermostat for large-scale
  {B}ayesian sampling}, in Advances in Neural Information Processing Systems,
  2015, pp.~37--45.

\bibitem{Talay2002}
{\sc D.~Talay}, {\em Stochastic {H}amiltonian systems: exponential convergence
  to the invariant measure, and discretization by the implicit {E}uler scheme},
  Markov Processes and Related Fields, 8 (2002), pp.~163--198.

\bibitem{varnai2013tests}
{\sc C.~V{\'a}rnai, N.~Bernstein, L.~Mones, and G.~Cs{\'a}nyi}, {\em Tests of
  an adaptive {QM/MM} calculation on free energy profiles of chemical reactions
  in solution}, The Journal of Physical Chemistry B, 117 (2013),
  pp.~12202--12211.

\bibitem{Villani2009}
{\sc C.~Villani}, {\em Hypocoercivity}, Mem. Amer. Math. Soc., 202 (2009).

\bibitem{WellingTeh2011}
{\sc M.~Welling and Y.~W. Teh}, {\em Bayesian learning via stochastic gradient
  {L}angevin dynamics}, in Proceedings of the 28th International Conference on
  Machine Learning (ICML-11), 2011, pp.~681--688.

\end{thebibliography}
